\documentclass[12pt]{amsart}
\usepackage[pagebackref=true,colorlinks,linkcolor=blue,citecolor=magenta]{hyperref}
\usepackage{amsmath,amssymb,latexsym} 
\usepackage{graphicx}
\usepackage{fullpage}
\usepackage{enumitem}
\usepackage[title]{appendix}
\usepackage{comment}
\usepackage{etoolbox}
\usepackage[square, numbers,sort&compress]{natbib}
\usepackage{hyperref}
\usepackage{cleveref}
\usepackage{array}
\usepackage{algorithmicx,algcompatible,algorithm}

\usepackage{multirow}

\usepackage{commands}

\theoremstyle{plain}
\newtheorem{theo}{Theorem}
\newtheorem{proposition}[theo]{Proposition}

\newtheorem{lemma}[theo]{Lemma}
\newtheorem{definition}{Definition}

\theoremstyle{remark}
\newtheorem{remark}{Remark}

\newtheorem{hyp}{Hypothesis}

\title{Ground staff shift planning under delay uncertainty at Air France}

\author{Julie Poullet}
\address{\'Ecole polytechnique, Route de Saclay, 91128 Palaiseau, France}
\email{frederic.meunier@enpc.fr}

\author{Axel Parmentier}
\address{A. Parmentier,
Universit\'e Paris Est, CERMICS, 77455 Marne-la-Vall\'ee CEDEX, France}
\email{axel.parmentier@enpc.fr}

\keywords{Stochastic ground staff scheduling, stochastic shift planning, column generation, stochastic resource constrained shortest path, flights delay}

\date{\today}

\begin{document}

\maketitle

\begin{abstract}
	Ground staff agents of airlines operate many jobs at airports such as passengers check-in, planes cleaning, etc. 
	Shift planning aims at building the sequences of jobs operated by ground staff agents, and have been widely studied given its impact on operating costs.
	As these jobs are closely related to flights arrivals and departures, flights delays disrupt ground staff schedules, which leads to high additional costs.
	Our goal is to provide a solution for shift planning at Air France that takes into account these additional costs.
	We therefore introduce a stochastic version of the shift planning problem that takes into accounts the cost of disruptions due to delay, and a column generation approach to solve it. 
	The key element of our column generation is the algorithm for the pricing subproblem, which we model as a stochastic resource constrained shortest path problem.
	Numerical results on Air France industrial instances prove the relevance of the shift planning problem and the efficiency of the solution method. 
	The column generation can solve to optimality instances with up to two hundred and fifty jobs.
	Moving from the deterministic problem to the stochastic one including delay costs enables to reduce the total operating costs by 3\% to 5\% on our instances.
\end{abstract}



\section{Introduction}

\subsection{Context}

Operations research has played an important role in airline management for more that fifty years \citep{barnhart2003applications}. 
Applications span a large spectrum of airlines activities, from revenue management to crew operations, air traffic management, etc. 
In this paper, we focus on \emph{ground staff shift planning}, which consists in building the schedules of ground staff agents.

During the last decades, the rise of global air traffic has been followed by an increase in flights delays. More than 20\% of European flights are delayed by more than fifteen minutes in 2005 \citep{santos2010determinants}.
Disruption of passenger connections, crews, and ground operations is a major source of cost for airlines \citep{cook2015european}.
More specifically, disruptions of ground operations have two negative consequences.
First, ground operations contribute to delay propagation: as a flight cannot take-off if mandatory jobs have not been performed by ground staff, 
a ground staff agent that is delayed on a given job propagates this delay to the flights corresponding to its next jobs.
And second, avoiding this propagation of delay requires additional staff to perform ground operations.
Building ground staffs schedules that are resilient with respect to delay is therefore growing challenge for airlines.

This paper is the result of a project initiated by Air France, the main French airline, to build such schedules.

\subsection{Literature review}
\label{sub:literature_review}


Many different kinds of jobs, also called tasks in the literature, are operated by airline ground staff agents in airports.
\citet[Chapter 1]{herbers2005models} describes these jobs and the contribution of operations research to their optimization.
If there is a large variety of jobs, from runway operations to plane cleaning, 
the same kind of optimization approaches can be applied to the different kind of jobs.
However, ad-hoc approaches are sometimes developed to deal with specific jobs such as check-in counters \citep{stolletz2010operational} and cargo facilities~\citep{yan2006long,yan2008stochastic,yan2008short,rong2009shift}.
Ground staff planning in airports is generally done in three stages. 
\emph{Task generation} takes in input the set of flights operated and generates the set of jobs that must be operated by ground staff agents.
\emph{Shift planning} then builds the sequences of jobs or \emph{shifts} that will be operated by agents on a given day. 
At this stage, shifts are not assigned to a specific agent but respect several working rules on working time and breaks.
Finally, \emph{rostering} combines these shifts into rosters spanning an horizon of several days and assign them to agents.
These stages tend to be integrated in recent contributions~\citep{brusco2011integrated,rong2009shift}.
If airplane and crew scheduling lead to problems specific to airlines \citep{barnhart2003applications}, the ground staff scheduling problems mentioned are more classical and well identified in the general literature on workforce scheduling. \citet{ernst2004staff} and~\citet{van2013personnel} provide detailed reviews of this literature.
As delay propagates along shifts operated on a given day, resilience of ground staff operations with respect to delay must be enforced during shift planning.

Shifts can be partitioned into different \emph{shift types}, which are identified by a starting time, an ending time, and different breaks.
Two kinds of approaches are used for shift planning. 
In \emph{demand-level shift planning}, jobs are affected to shift types, but the shifts, i.e.~the sequences of jobs operated for each shift type are not built. This problem is known as the \emph{shift scheduling} problem \citep{dantzig1954letter} in the general workforce literature \citep{ernst2004staff}.
On the contrary, \emph{task-level shift planning} builds the shifts. This problem integrates shift scheduling and interval scheduling \citep{kolen2007interval} problems of the general workforce literature~\citep{ernst2004staff}.
Task-level is more precise that demand-level shift planning, but it also leads to more involved optimization problems.
As the cost of a shift generally depends only on its type, demand-level shift planning is the most frequently used in the general workforce literature and for ground staff scheduling \citep{brusco1995improving,holloran1986united,schindler1993station}. Jobs are then assigned shortly before operations \citep{dowling1997staff}.
However, task-level shift planning has also been proposed for ground staff scheduling \citep{stolletz2010operational}.
As delay propagates along shifts, demand-level shift planning cannot deal with delay propagation, and we use a task-level shift planning approach. 

Mathematical programming, heuristics, and constraint programming are the main solution methods used in workforce optimization~\citep{van2013personnel}, the two first ones being the most frequent. 
For ground staff scheduling, approaches rely on mixed-integer programming~\citep{holloran1986united,rong2009shift,schindler1993station,stolletz2010operational,yan2006long,yan2008short,yan2008stochastic}, matheuristics based on column generation and local search~\citep{brusco1995improving}, and tabu search~\citep{brusco2011integrated}. 
Decomposition methods and column generation approaches such as the one we use are frequent in the general workforce literature~\citep{van2013personnel} but not in ground staff scheduling.

Uncertainty models on demand volume, demand arrival time, and manpower availability have been proposed in personnel scheduling problems~\citep{van2013personnel}. 
Uncertainty in demand has been considered in cargo ground staff scheduling \citep{yan2008stochastic,yan2008short}, and resilience with respect to delay of airplane and crew schedules 
\citep{duck2012increasing,lan2006planning,weide2010iterative,dunbar2014integrated,yen2006stochastic,yan2016robust,ageeva2000approaches}, using stochastic and robust optimization approaches. 
\citep{wu2016airline} has proposed a model for delay propagation taken into account ground staff shifts.
But to the best of our knowledge, no ground staff shift planning approach taking into account delay propagation has been proposed.

\subsection{Contribution}

The first contribution of this paper is a \emph{stochastic ground staff shift planning problem} that can deal with Air France specific context.
The framework is flexible enough to deal with the different kind of jobs operated by ground agents of the company. 
The way delay is propagated along shifts, the online management of delay, and the resulting costs for the airline have been modeled with care.
We notably introduce a modeling hypothesis on delay propagation that fits the way the airline handles delay and simplifies optimization. Scenario based distributions are considered in solution methods.




Second, we provide compact integer formulation for the deterministic task-level shift planning and generalize it to the stochastic task-level shift planning. If similar formulations of the deterministic problem have been considered in the literature, our formulation is tailored for Air France specific problem, and off-the-shelf solvers are able to solve to optimality instances with hundreds of jobs in a few seconds. The stochastic version can solve heuristically instances with a few dozens of jobs, and highlights the difficulty of the stochastic shift scheduling problem.

Our main contribution is a column generation approach to the stochastic shift planning problem that can solve to optimality instances with up to 210 jobs and 200 scenarios in a few minutes, and instances with up to 250 jobs and 200 scenarios in at most a few hours on a standard computer.
To the best of our knowledge, this is the first practically efficient algorithm for shift planning with delay propagation.
The pricing subproblem of the column generation build interesting shifts, and the master problem selects the shifts that will be operated.
As delay propagates along shifts, stochasticity is handled in the pricing subproblem.
The performance of our approach lies in the efficiency of the pricing subproblem algorithm.
We solve this pricing subproblem using a framework for resource constrained shortest path recently introduced by the second author \citep{parmentier2015algorithms}.
We emphasize the the framework does not explain how to model a concrete problem in practice.
Indeed, to successfully apply the framework, we had to introduce a non-trivial algebraic structure to model delay propagation. 
This algebraic structure provides insights on delay propagation and is a contribution on its own.
Furthermore, the framework had never been applied to a concrete stochastic industrial problem, and our work demonstrates the relevance of such an approach.

Finally, we demonstrate the practical relevance of a stochastic approach including the cost of delay: moving from the deterministic problem to the stochastic one enables to reduce the total operating costs by 3\% to 5\% on our instances.


\subsection{Organization of the paper}
The deterministic and stochastic shift planning problems are introduced in Section~\ref{sec:ProblemStatement}. 
Section~\ref{sec:compactMIP} provides the compact integer formulation as well as numerical experiments on industrial instances. 
The column generation approach and numerical experiments showing its performance on industrial instances are detailed in Section \ref{sec:CGApproach}. These experiments show the relevance of stochastic shift planning. The algorithm solving the pricing subproblem of the column generation is presented in Section \ref{sec:PrincingSubProblem}. All proofs are postponed to the appendix.

\section{Problem statement}
\label{sec:ProblemStatement}
We now introduce our ground staff shift planning problem.

\subsection{Jobs, shifts, and lunch breaks}
\label{sub:jobs_shifts_and_lunch_breaks}

Ground staff shift planning consists in assigning a given set of \emph{jobs} to teams of skilled agents. 
This set of jobs may be composed of various jobs, such as ensuring baggage check-in, helping to print baggage tag at the interactive kiosk, or being in charge of the boarding. 
Each team starts at a given time,
operates jobs, possibly making a scheduled lunch break in-between two jobs, and finishes at a given time.
This sequence of activities is called a shift.
Roughly speaking, the ground staff shift planning consists in building the shifts, that is, the sequences of jobs operated by the staff teams, in order to operate the jobs at minimum cost.

More formally, we denote by $\calHb$ the finite set of times at which staff members can start working, and by $\calHf$ the finite set of time at which they can stop working.
Let $J$ be the set of \emph{jobs} to be scheduled, and $\brSet$ be the set of possible lunch \emph{breaks}.
An \emph{activity} $j$ is either a job in $J$ or a break in~$B$.
It is characterized by a \emph{fixed time interval} [$\hbtask$,$\hftask$], where $\hbtask$ represents its \emph{beginning time} and $\hftask$ its \emph{ending time}.
In our case, breaks are scheduled between $\tbl$ and $\tel$ and last $\tbr$, hence~$B = \big\{[t,t+\tbr]\colon t \in \{\tbl, \tbl +1 ,\ldots, \tel - \tbr\}\big\}$.

\begin{definition} \label{def:shifts}
A \emph{shift} $\shift$ is 
 a sequence $\tb[\sh],\task^{1},\task^{2}, \ldots, \task^{k}, \te[\sh]$
 such that $\tb[\sh] \in \calHb$, $\te[\sh] \in \calHf$, activity $j^i$ belong to $J \cup \brSet$ for each $i$ in $\{1,\ldots,k\}$, and
\begin{equation*}
    \hbshift \leq\tb[\task^{1}], \quad 
    \te[j^{k}] \leq \hfshift, \quad \text{and} \quad
    \te[j^i] \leq \tb[j^{i+1}] \text{ for all $i$ in } \{1,\ldots,k-1\}.
\end{equation*}
Integers $\tb[\sh]$ and $\te[\sh]$ are respectively the \emph{beginning time} and \emph{ending time} of $\sh$.
\end{definition}

To be operated, a shift must satisfy some working rules of the airline. First, its duration must be non greater than the \emph{maximum duration} $\tm$ in $\bbZ_+$. Second, if the intersection of the shift with $[\tbl, \tel]$ is longer that $\tml$, 
then the staff must be able to take a break. $\tbl$, $\tel$, and $\tml$ are constants in $\bbZ_+$.
More formally, a shift  $\sh$ is \emph{feasible} if it satisfies the following  rules.

\begin{enumerate}[label={(\arabic*)}]
    \item \label{rule:Amplitude} $\hfshift - \hbshift \leq \tm$. 
        \item \label{rule:LunchBreak} if $|[b_{lunch}, f_{lunch}] \cap [\hbshift,\hfshift]| \geq \tml$ then $\brSet \cap \sh \neq \emptyset$.
\end{enumerate}
Figure \ref{examplePath} illustrates two shifts, one of them containing a lunch break. 
Yellow rectangles illustrate the shifts time intervals, black lines the jobs time intervals, and blue lines the lunch breaks.

\begin{figure}[h]
    \centering
    \begin{tikzpicture}[scale = 1.5]
         
        
        \draw[line width=0.4mm] (2.4,-2) --(4,-2);
        \draw[line width=0.4mm] (5.5,-2) --(6,-2);
        \draw[line width=0.4mm] (6.2,-2) --(7,-2);
        
        \draw[line width=0.4mm] (5.4,-3) --(6,-3);
        \draw[line width=0.4mm] (6.5,-3) --(8,-3);
        \draw[line width=0.4mm] (8.2,-3) --(8.5,-3);
        \draw[line width=0.4mm] (8.6,-3) --(9,-3);
        \draw[line width=0.4mm] (9.2,-3) --(9.8,-3);

        \draw[dashed] (3.5,-1.2) to ++ (0,-2.2);
        \draw[dashed] (5.3,-1.2) -- ++ (0,-2.2);
        \node[right] at (3.5,-1.2) {$\tbl$};
        \node[right] at (5.3,-1.2) {$\tel$};
        
        
        \draw[line width=0.4mm,blue] (4.2,-2) --(5.2,-2);
        
        \draw[opacity =0.2,fill = yellow] (1.8,-1.7)--(7.2,-1.7)--(7.2,-2.3)--(1.8,-2.3)--(1.8,-1.7);
        \draw[opacity =0.2,fill = yellow] (5.2,-2.7)--(10,-2.7)--(10,-3.3)--(5.2,-3.3)--(5.2,-2.7);

        \node[above] at (1.8,-1.7) {$\hb[\sh_1]$};
        \node[above] at (7.2,-1.7) {$\he[\sh_1]$};
        \node[below] at (5.2,-3.3) {$\hb[\sh_2]$};
        \node[below] at (10,-3.3) {$\he[\sh_2]$};
        \node (agent2) at (1.2,-2) {$\sh_1$};
        \node (agent3) at (1.2,-3) {$\sh_2$};
        
        \node[draw,rounded corners =0.10cm, align=left] at (9,-2) {\rule{0.5cm}{1pt} jobs in $J$\\\textcolor{blue}{\rule{0.5cm}{1pt} breaks in $B$}};
    
    \end{tikzpicture} 
    \caption{Example of building shifts covering all the jobs and including lunch break}
    \label{examplePath} 
\end{figure}
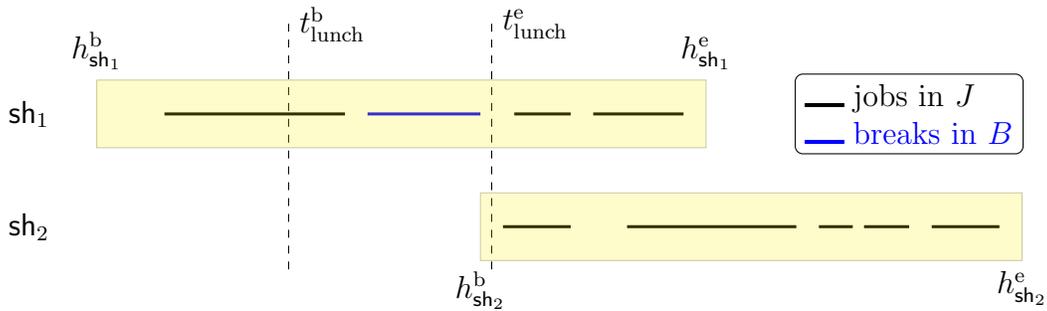

\subsection{Delay and back-up agents}
\label{sub:delay_and_back_up_agents}


\paragraph{}
When a flight is delayed, the delay is propagated to the jobs associated with the flight. Depending on the type of the job, delay may influence the beginning time, the ending time (e.g. ensuring check-in), or both (e.g. cleaning plane), causing potential issues in the schedule of the agent realizing them as some job successions may not be possible anymore. 
At Air France, in that situation, \emph{back-up agents} are sent by the operations' manager to realize all the jobs that the initial agent cannot realize. 
A job is said to be \emph{rescheduled} if it is operated by a back-up agent.
As rescheduling a job is costly, the airline wants to minimize the number of rescheduled jobs.
A job is rescheduled if~\ref{case:succ} the agent who is supposed to operate it is still doing a previous job when the job starts, if~\ref{case:break} operating it would prevent the agent from taking a lunch break, or if~\ref{case:late} the job is very late. 
The rationale behind this third condition is the following: 
we want to avoid rescheduling due to small delays, but not rescheduling due to large delays.
Indeed, large delay are rare in practice, and it makes sense to reschedule very late jobs, as we would otherwise obtain shifts with unnecessary long slacks between jobs.
Based on~\ref{case:late}, we make the following simplifying assumption.



\begin{hyp}\label{hyp:propagation}
Delay propagates only to the next job under condition~\ref{case:succ} and~(ii).
\end{hyp}
    
\noindent In other words, a job in a shift $\sh$ can be rescheduled due to the job right before in $\sh$, but not due to the previous ones.
Let $\xb[j]$ and $\xe[j]$ be the realized beginning and ending time of a job $j$. 
Both $\xb[j]$ and $\xe[j]$ are random variables from the set of scenarios $\Omega$ to $\bbZ_+$.
Let $\vl[j]$ be a binary random variable indicating if job $j$ is very late. 
Job $j^i$ in a shift $\sh$ is operated by a back-up agent if one of the following conditions is satisfied.
\begin{enumerate}[label={(\roman*)}]
    \item \label{case:succ} job $j^{i-1}$ is operated before $j^i$ and not by a back-up agent, and $\xe[j^{i-1}] > \xb[j^i]$.
    \item \label{case:break} job $j^{i-2}$ and lunch break $j^{i-1}$ are operated before $j^i$ in $\sh$, and $\xe[j^{i-1}] > \xb[j^i] - \tbr$.
    \item \label{case:late} $\vl[j] = 1$.
    %
\end{enumerate}

Figure~\ref{fig:DelayExample} illustrates how these conditions work under several delay scenarios. 
Under scenario $\omega_1$, job $j^2$ is operated by a back-up agent due to condition~\ref{case:late}, while under scenario $\omega_2$, a back-up agent is needed for $j^3$ due to~\ref{case:succ}. Under scenario $\omega_3$, job $j^2$ and $j^3$ still cannot be operated by the same agent. 
But as $j^2$ is operated by a back-up agent, $j^3$ can be operated by the initial agent.


  \begin{figure}[h]
  \centering
   \begin{tikzpicture}[scale = 1.5]
        \tikzset{
            job/.style={thick},
            arc/.style={->, thick, draw=green, dotted},
            ome/.style={red},
            bu/.style={draw=blue, dashed}
        }

         \def\l{0.8}
         \def\h{-1.2}

        \node at (1*\l,0*\h) {No delay};

        \node (b1a) at (2*\l,0*\h) {};
        \node (e1a) at (3*\l,0*\h) {};
        \node (b2a) at (3.5*\l,0.3*\h) {};
        \node (e2a) at (6*\l,0.3*\h) {};
        \node (b3a) at (6.4*\l,0*\h) {};
        \node (e3a) at (8*\l,0*\h) {};

        \draw[job] (b1a.center) to node[midway, above] {$j^1$} (e1a.center);
        \draw[job] (b2a.center) to node[midway, below] {$j^2$} (e2a.center);
        \draw[job] (b3a.center) to node[midway, above] {$j^3$} (e3a.center);
        \draw[arc] (e1a.center) -- (b2a.center);
        \draw[arc] (e2a.center) -- (b3a.center);

        \node at (1*\l,1*\h) {Scenario $\omega_1$};

        \node (b1a) at (2*\l,1*\h) {};
        \node (e1a) at (3*\l,1*\h) {};
        \node (b2a) at (3.5*\l+2.2*\l,1.3*\h) {};
        \node (e2a) at (6*\l+2.2*\l,1.3*\h) {};
        \node (b3a) at (6.4*\l,1*\h) {};
        \node (e3a) at (8*\l,1*\h) {};

        \draw[job] (b1a.center) to node[midway, above] {$j^1$} (e1a.center);
        \draw[job,bu] (b2a.center) to node[midway, below] {$j^2$} (e2a.center);
        \draw[job] (b3a.center) to node[midway, above] {$j^3$} (e3a.center);
        \draw[arc] (e1a.center) to[bend left=10] (b3a.center);

        \node at (1*\l,2*\h) {Scenario $\omega_2$};

        \node (b1a) at (2*\l,2*\h) {};
        \node (e1a) at (3*\l,2*\h) {};
        \node (b2a) at (4.2*\l,2.3*\h) {};
        \node (e2a) at (6.7*\l,2.3*\h) {};
        \node (b3a) at (6.4*\l,2*\h) {};
        \node (e3a) at (8*\l,2*\h) {};

        \draw[job] (b1a.center) to node[midway, above] {$j^1$} (e1a.center);
        \draw[job] (b2a.center) to node[midway, below] {$j^2$} (e2a.center);
        \draw[job,bu] (b3a.center) to node[midway, above] {$j^3$} (e3a.center);
        \draw[arc] (e1a.center) to (b2a.center);
        \draw[arc,ome] (e2a.center) to node[midway]{/} (b3a.center);

        \node at (1*\l,3*\h) {Scenario $\omega_3$};

        \node (b1a) at (2.9*\l,3*\h) {};
        \node (e1a) at (4.5*\l,3*\h) {};
        \node (b2a) at (4.2*\l,3.3*\h) {};
        \node (e2a) at (6.7*\l,3.3*\h) {};
        \node (b3a) at (6.4*\l,3*\h) {};
        \node (e3a) at (8*\l,3*\h) {};

        \draw[job] (b1a.center) to node[midway, above] {$j^1$} (e1a.center);
        \draw[job,bu] (b2a.center) to node[midway, below] {$j^2$} (e2a.center);
        \draw[job] (b3a.center) to node[midway, above] {$j^3$} (e3a.center);
        \draw[arc,ome] (e1a.center) to node[midway]{$\backslash$} (b2a.center);
        \draw[arc,ome] (e2a.center) to node[midway]{/} (b3a.center);
        \draw[arc] (e1a.center) to[bend left=10] (b3a.center);

        \node[draw,rounded corners = 0.10cm, align = left] at (10.5*\l,-2) {
        \tikz[baseline=-0.5ex]{\draw[job] (0,0) -- (0.7,0);} initial agent \\ 
        \tikz[baseline=-0.5ex]{\draw[job,bu] (0,0) -- (0.7,0);} back-up agent\\
        \tikz[baseline=-0.5ex]{\draw[arc] (0,0) -- (0.7,0);} operated succession \\
        \tikz[baseline=-0.5ex]{\draw[arc,ome] (0,0) to node[midway]{/} (0.7,0);} broken succession 
        };
    \end{tikzpicture} 
    \caption{Example of delays requiring a backup agent}
    \label{fig:DelayExample}
\end{figure}
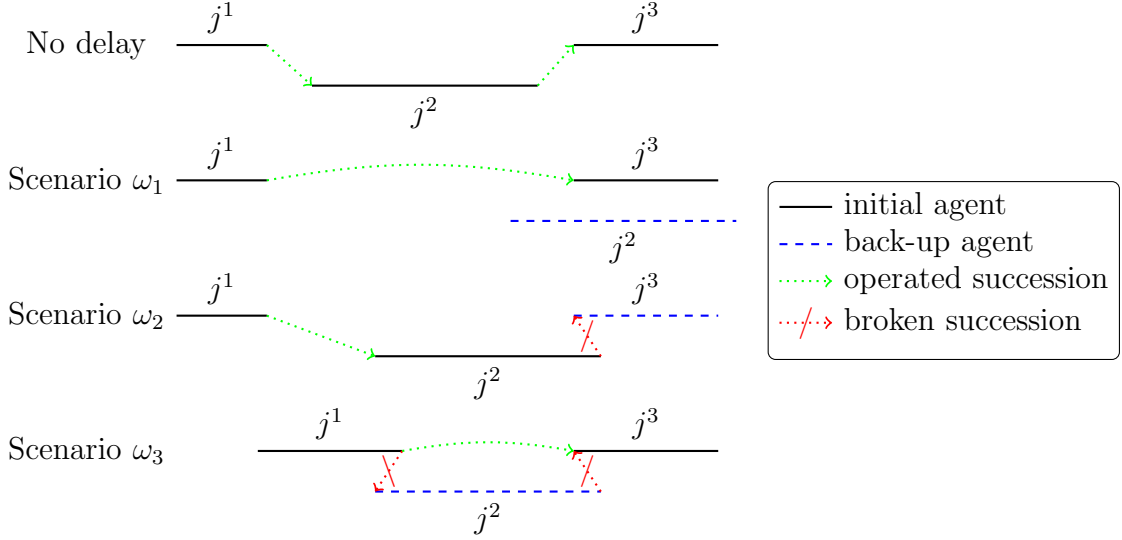

\begin{remark}
Our approach applies to different kinds of ground staff jobs.
The definition of a very late job, that is, of $\vl[j]$, depends on the kind of job $j$ that is operated. The duration of some jobs such as cleaning the plane is deterministic while the duration of others such as boarding passengers is random. 
Some operations realized on the planes cannot be interrupted for security reasons, while some other jobs such as being at the checking-board can be operated by several successive teams of agents.
For clarity, we consider here variables $\vl[j]$ as an input.



\end{remark}

\subsection{Costs}
\label{sub:costs}

Operating a shift $\sh$ generates two sources of costs for the airline.
The first one comes from agents wage and is a non-decreasing function of the total duration $\cw(\he[\sh] - \hb[\sh])$. Function $\cw(\cdot)$ is illustrated on Figure~\ref{fig:costOfShift}. 
The second one is due to a fixed cost $\cbu$ incurred for each rescheduled job.
Let $\nbu[\sh]$ be the (random) number of jobs of $\sh$ rescheduled.
The expected cost $c_{\sh}$ of a shift is
\begin{equation}\label{eq:shiftCost}
    c_{\sh} = \cw(\he[\sh] - \hb[\sh]) + \bbE\big(\cbu \nbu[sh]\big).
\end{equation}
In our results, $\cbu$ is equal to two hours of a regular agent.


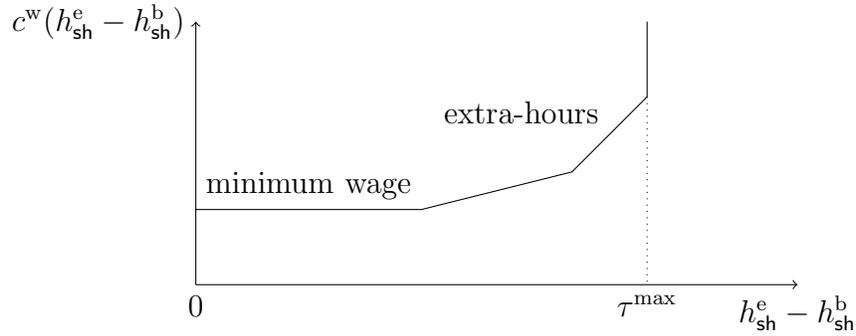
\begin{figure}
\begin{center}

\begin{tikzpicture}

\def\l{1}
\def\h{0.5}

\draw[->] (0*\l,0*\h) node[below] {0} -- (8*\l,0*\h) node[below]{$\he[\sh] - \hb[\sh]$};
\draw[->] (0*\l,0*\h) -- (0*\l,7*\h) node[left] {$\cw(\he[\sh] - \hb[\sh])$};

\draw (0*\l,2*\h) -- node[midway,above] {minimum wage} (3*\l,2*\h) -- (5*\l,3*\h) -- node[midway, above left] {extra-hours} (6*\l,5*\h) -- (6*\l,7*\h);
\draw[dotted] (6*\l,5*\h) to (6*\l,0*\h) node[below] {$\tm$};
\end{tikzpicture}
\end{center}
\caption{Cost of a shift as a function of duration.}
\label{fig:costOfShift}
\end{figure}


\subsection{Well-scheduled breaks}
\label{sub:well_scheduled_breaks}

A feasible $\sh$ is \emph{well-scheduled} if it contains at most one break, and any break $j$ in $\sh$ satisfies
\begin{align*}
\te[j] = \left\{ 
 \begin{array}{ll}
\min(\te[\sh],\tel) \enskip& \text{if $j$ is the last activity of $\sh$,} \\
 \min(\tb[j'],\tel) & \text{otherwise, where $j'$ is the activity after $j$ in $\sh$.}
 \end{array}\right.
\end{align*} 
We denote by $\calSH$ the set of well scheduled shifts.
The following proposition shows that we can restrict ourselves to well-scheduled shifts, and is proved in appendix.

\begin{proposition}\label{prop:wellScheduled}
Given an arbitrary feasible shift $\sh$, there exists a well-scheduled shift $\sh'$ containing the same jobs and such that $c_{\sh'} \leq c_{\sh}$.
\end{proposition}


\subsection{Problem statement}
\label{sub:problem_statement}

The \emph{stochastic ground staff shift planning problem} \eqref{pb} consists in finding a set $\calS$ of well-scheduled shifts of minimum cost and such that each job in $J$ is operated by at least one shift in $\calS$
\begin{equation}\label{pb}\tag{PB}
    \min \bigg\{\sum_{\sh \in \calS} c_{\sh} \colon  \calS \subseteq \calSH \text{ and there is $\sh \in \calS$ such that $j \in \sh$ for each }j \in J \bigg\}.
\end{equation}

\section{Compact Mixed Integer Linear Program}
\label{sec:compactMIP}

\subsection{Shift digraph}
\label{sub:shift_digraph}

We first explain how well-scheduled shifts can be modeled as paths in an acyclic directed graph. This will make the description of the integer program a straightforward job. 
Let $\{\mathsf{bl},\mathsf{al}\}$ be a set of two elements, $\mathsf{bl}$ and $\mathsf{al}$ respectively stand for before lunch and after lunch.
Let
\begin{alignat*}{5}
    \Jbl[ \hb]= \left\{j\in J  
    \left| \begin{array}{l}
    \tb[j] \geq  \hb \\ 
     \te[j] \leq \tel - \tbr \\
     \te[j]\leq \hb + \tm
    \end{array}\right.\right\}
    \quad \text{and} \quad 
    \Jal[\hb]= \left\{j\in J 
    \left|
    \begin{array}{l}
    \tb[j] \geq  \hb \\
    \tb[j] \geq \tbl + \tbr\\
    \te[j]\leq \hb + \tm
    \end{array}
    \right.\right\}
\end{alignat*}
be respectively the set of jobs that can be operated before or after the lunch break in a shift starting at $\hb$. 
The set of vertices is
\begin{equation*}
    V = \{o\} \cup \left(\bigcup_{\hb \in \calHb} \Jbl[\hb] \times \{\hb \} \times \{\mathsf{bl}\}\right) 
    \cup \left(\bigcup_{\hb \in \calHb}\Jal[\hb] \times \{\hb\} \times \{\mathsf{al}\} \right) \cup \calHf \cup \{d\},
\end{equation*}
where $o$ and $d$ are respectively an origin and a destination vertex.
The set of arcs $A$ contains the following ordered pairs of $V\times V$.
\begin{itemize}
    \item $\big(o,(j,\hb,\bl)\big)$ for each $\hb \leq \tel - \tml$ in $\calHb$ and $j$ in $\Jbl[\hb]$.
    \item $\big(o, (j,\hb,\al)\big)$ for each $\hb > \tel - \tml$ in $\calHb$ and $j$ in $\Jal[\hb]$.
    \item $\big((j,\hb,\bl), (j',\hb,\bl)\big)$ for each $\hb\in\calHb$ and $j$, $j'$ in $\Jbl[\hb]$ such that $\te[j] \leq \tb[j']$.
    \item $\big((j,\hb,\bl), (j',\hb,\al)\big)$ for each $\hb\in\calHb$, $j$ in $\Jbl[\hb]$, and $j'$ in $\Jal[\hb]$ such that $\te[j] + \tbr \leq \tb[j']$.
    \item $\big((j,\hb,\al), (j',\hb,\al)\big)$ for each $\hb\in\calHb$ and $j$, $j'$ in $\Jal[\hb]$ such that $\te[j] \leq \tb[j']$.
    \item $\big((j,\hb,\bl),\he\big)$ for each $\he < \tbl + \tml$ in $\calHf$ and $j$ in $J$ satisfying $\te[j] \leq \he \leq \hb + \tm$.
    \item $\big((j,\hb,\bl),\he\big)$ for each $\he \geq \tbl + \tml$ in $\calHf$ and $j$ in $J$ satisfying $\te[j] + \tbr \leq \he \leq \hb + \tm$.
    \item $\big((j,\hb,\al),\he\big)$ for each $\he \geq \tbl + \tml$ in $\calHf$ and $j$ in $J$ satisfying $\te[j] \leq \he \leq \hb + \tm$.
    \item $\big(\he,d\big)$ for each $\hb$ in $\calHf$.
\end{itemize}
For each $j \in J$, let $V_j = \big\{(j,\hb,\el) \in V\colon \hb \in \calHb \text{ and } \el \in\{\bl,\al\}\big\}$, and for each arc $a$ in $A$, let 
$$ c_a = 
 \left\{\begin{array}{ll}
 \cw(\he - \hb) \enskip & \text{if $a$ is of the form $\big((j,\hb,\el),\he\big)$,} \\
 0 & \text{otherwise.}
 \end{array}\right.
 $$ 

\begin{proposition}\label{prop:bijectionDigraph}
There is a bijection between well-scheduled shifts and $o$-$d$ paths in $D$, which associates to a well-scheduled shift $\sh$ an $o$-$d$ path $P$ such that a job $j $ is in $\sh$ if and only if $P$ intersects $V_j$.
\end{proposition}

\begin{figure}[h]
    \centering
    \begin{tikzpicture}[roundnode/.style={circle, draw=green!60, fill=green!5, very thick, minimum size=7mm},
roundnodeB/.style={rectangle, draw=blue!60, fill=blue!5, very thick, minimum size=3mm},
squarenode/.style={rectangle, draw=red!60, fill=red!5, very thick, minimum size=5mm},
]
        \node[squarenode] (initial) at (1,1.5,0) {o};
        \node (debut1) at (2,-1,0) {}; 
        \node (debut2) at (2,3,0) {}; 
        \node at (1.5,3,-0.5) {$\hb[2]$};
        \node at (1.5,-1,0) {$\hb[1]$};
        \node[roundnode] (n1) at (4,-1,-1) {$j_1$};
        \node[roundnode] (n21) at (4,-1,1) {$j_2$};
        \node[roundnode] (n22) at (4,3,1) {$j_2$};
        \node[roundnode] (n31) at (6,0,0) {$j_3$};
        \node[roundnode] (n32) at (6,4,-1) {$j_3$};
        \node[roundnode] (n4) at (8,4,-1) {$j_4$};
        \node[roundnodeB] (fin1) at (10,0.7,0) {$\he[1]$};
        \node[roundnodeB] (fin2) at (10,2.3,0) {$\he[2]$};
        \node[squarenode] (fin) at (11,1.5,0) {d};
         
        \draw[->] (initial.east) .. controls  (3,-1,-0.8)..(n1.west);
        \draw[->] (initial.east) .. controls  (3,-1,1.8)..(n21.west);
        \draw[->] (initial.east) .. controls (5,0,0).. (n31);
        \draw[->] (n1.east) -- (n31.south);
        \draw[->] (n1.east) .. controls (8,-1,0) .. (fin1.west);
        \draw[->] (n21.east) .. controls (5,-1,0.65) .. (n31.south);
        \draw[->] (n21.east) .. controls (8,-1,1.5) .. (fin1.west);
        \draw[->] (n31.east) -- (fin1.west);
        \draw[->] (n31.east) -- (fin2.west);
        \draw[->] (n1.east) .. controls (7,-1,-1.5).. (fin2.west); 
        \draw[->] (n21.east) .. controls (7,-1,1)..(fin2.west);
        \draw[->] (fin1.east) -- (fin.south);
        
        \draw[red,->] (initial.east) .. controls (3,3,0.7) .. (n22.west);
        \draw[->] (initial.east) .. controls (4,4,-1.2) .. (n32.west); 
        \draw[->] (initial.east).. controls (4,5.5,-1.2).. (n4.west);
        \draw[red,->] (n22.east) -- (n32.west);
        \draw[->] (n22.east) .. controls (6,3,1.2) .. (n4.west);
        \draw[->] (n22.east) .. controls (8,3,2) .. (fin2.west);
        \draw[red,->] (n32.east) -- (n4.west);
        \draw[->] (n32.east) .. controls (8.5,5,-1.4) .. (fin2.west);
        \draw[->] (n32.east) -- (fin1.west);
        \draw[->] (n22.east) .. controls (8,3,0).. (fin1.west);
        \draw[red,->] (n4.east) -- (fin2.west);
        \draw[red,->] (fin2.east)-- (fin.north);
        
        \draw[opacity =0.2,fill = yellow] (1,-1,-1.5) -- (6,-1,-1.5)--(6,-1,1.5)--(1,-1,1.5)--(1,-1,-1.5);
        \draw[opacity =0.2,fill = orange] (5,0,-1.2) -- (10,0,-1.2)--(10,0,1.5)--(5,0,1.5)--(5,0,-1.2);
        \draw[opacity =0.2,fill = yellow] (1,3,-1.5) -- (6,3,-1.5)--(6,3,1.5)--(1,3,1.5)--(1,3,-1.5);
        \draw[opacity =0.2,fill = orange] (5,4,-1.5) -- (10,4,-1.5)--(10,4,1.5)--(5,4,1.5)--(5,4,-1.5);
    \end{tikzpicture}
    \caption{Example of modeling with four feasible time intervals}
    \label{fig:ModelingMIP}
\end{figure}
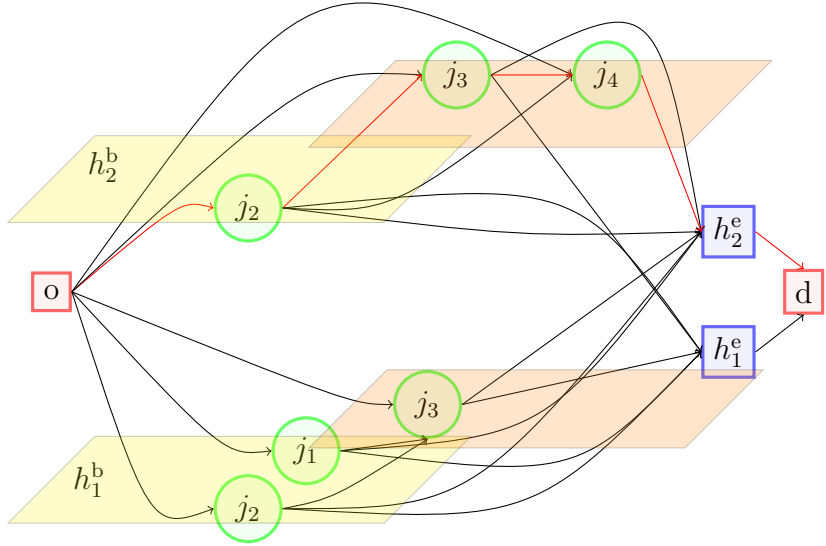
The proof of Proposition~\ref{prop:bijectionDigraph} is given in appendix. We explain here the semantic of our digraph on the example illustrated on Figure \ref{fig:ModelingMIP}.
Consider the $o$-$d$ path $P$ in red on Figure~\ref{fig:ModelingMIP}.
According to Proposition~\ref{prop:bijectionDigraph}, there is a unique well-scheduled shift $\sh = \tb[\sh],\task^{1},\task^{2}, \ldots, \task^{k}, \te[\sh]$ that corresponds to $P$. 
We are now going to describe it.
For each beginning time $\hb$ in $\calHb$, let $V_{\hb} = \big(\Jbl[\hb] \times \{\hb \} \times \{\mathsf{bl}\}\big) \cup \big(\Jal[\hb] \times \{\hb\} \times \{\mathsf{al}\}\big)$. 
Remark that if $\hb[i] \neq \hb[i']$, there is no path from $V_{\hb[i]}$ to $V_{\hb[i']}$.
Hence, $P$ intersects a unique $V_{\hb}$. On Figure~\ref{fig:ModelingMIP}, we can see that it is $V_{\hb[2]}$.
This indicates that $\tb[\sh] = \hb[2]$.
The sequence of vertices $(j,\hb,\cdot)$ in $P$ indicates the successive jobs in $\sh$.
If $(j,\hb,\bl)$ is in $P$, it indicates that job $j$ is operated before the lunch break in $\sh$, and if $(j,\hb,\al)$ is in $P$, that $j$ is operated after the lunch break in $\sh$.
On Figure~\ref{fig:ModelingMIP}, the vertices in the yellow rectangles correspond to jobs operated before the lunch break, and the orange rectangles to those operated after lunch break.
On our red path $P$, three jobs are operated, $j_2$, $j_3$, and $j_4$, and a lunch break is taken between $j_2$ and $j_3$.
The penultimate vertex of an $o$-$d$ path $P$ is a vertex $\he \in \calHf$. It indicates that the ending time $\te[\sh]$ is equal to $\he $. 
For instance, on our red path, $\te[\sh] = \he[2]$. 
The proof of Proposition~\ref{prop:bijectionDigraph} shows that, given an $o$-$d$ path $P$, the shift $\sh$ constructed using the method described satisfies rules \ref{rule:Amplitude} and \ref{rule:LunchBreak}.







\subsection{Mixed integer program}
\label{sub:mixed_integer_program}

Proposition~\ref{prop:bijectionDigraph} ensures that a solution of the ground staff shift planning problem is a collection of $o$-$d$ paths in $D$. 
For each arc $(u,v)$, let $M_{(u,v),\omega}$ be a binary equal to one if the job of $v$ cannot be operated by the same team as the job of $u$ under scenario $\omega$, that is, if
$u = (j,\hb,\el)$ and $v = (j',\hb,\el)$ for some $\hb \in \calHb$ and $\el \in \{\bl,\al\}$, and $\xe[j](\omega) > \xb[j'](\omega)$, 
or $u = (j,h,\bl)$ and $v = (j',h,\al)$ for some $h \in \calHb$, and $\xe[j](\omega) + \tbr > \xb[j'](\omega)$, and to zero otherwise. 
And let $M_{v,\omega}$ be a binary equal to one if $v \in V_j$ and $\vl[j]=1$, and to zero otherwise. 
Consider the following MILP.
\begin{subequations}\label{eq:PLNECompact}
    \begin{alignat}{2}
    \min_x \enskip&\sum_{a \in A} c_{a}x_{a} + \frac{\cbu}{|\Omega|} \sum_{v \in V}\sum_{\omega \in \Omega } y_{v}^{\oscenario} 
    &\quad &  \label{objectiv} \\
    \mathrm{s.t.}\enskip
    & \sum_{v \in V_j}\sum_{a \in \delta^{-}(v)}x_{a} = 1  &&  \forall  j \in J \label{cst:cover}\\
    & \sum_{a \in \delta^{-}(v)}x_{a} = \sum_{a \in \delta^{+}(v)}x_{a}  && \forall v \in V \backslash\{o,d\} \label{cst:flow}\\
    & y_{v,\omega} \geq M_{(u,v),\omega} x_{(u,v)} - y_{u,\omega} && \forall (u,v) \in A, \enskip \forall \omega \in \Omega \label{cst:delayArc}\\
    & x_{a} \in \{0,1\}  && \forall a \in A \label{cst:arcIntegrity} \\
    &  y_{v,\omega} \in \{0,1\}, \enskip y_{v,\omega}\geq M_{v,\omega}  &&   \forall v \in V, \enskip \forall \oscenario \in \Oscenario \label{cst:delayIntegrity} 
    \end{alignat}
\end{subequations}

The binary variable $x_a$ indicates if arc $a$ is in the solution. Constraint~\eqref{cst:flow} is a flow constraint, and together with the integrality of the $x_a$, it ensures that the solution of \eqref{eq:PLNECompact} is a collection of $o$-$d$ path in $D$. 
Constraint~\eqref{cst:cover} ensures that each job in $J$ is covered by a unique $o$-$d$ path in the solution.
The binary variable $y_{v,\omega}$ indicates if $v$ is on a path in the solution and the job of $v$ in the corresponding shift is operated by a back-up team under scenario $\omega$.
Indeed, Constraint~\eqref{cst:delayIntegrity} ensures that $y_{v,\omega} = 1$ if \ref{case:late} is satisfied, and Constraint~\eqref{cst:delayArc} is a big M constraint ensuring that $y_{v,\omega} = 1$ if~\ref{case:succ} or~\ref{case:break} are satisfied. 
The following Proposition is therefore an immediate corollary of Proposition~\ref{prop:bijectionDigraph}.

\begin{proposition}
MILP~\eqref{eq:PLNECompact} is equivalent to the ground staff shift planning problem.
\end{proposition}

\subsection{Numerical results}
\label{sub:numerical_results}


\subsubsection{Experimental settings.}
\label{subsubsec:ExperimentalSettings}
All the experiments have been performed on a Linux computer with 4 cores at 3.8GHz, and 8.1 Gb of RAM. The algorithms are not parallelized. \texttt{GUROBI 7.0.2} is used to solve all linear and integer programs. 
 
 \subsubsection{Instances description.}
 \label{subsubsec:AFInstances}
The instances used are real instances from runway jobs operated by Air France staff. 
Table \ref{tab:AFInstances} describes eight of these instances, whose identifying number is provided in the first column.
The first column gives instances name, and the second column the number of jobs in the instance.
The third column gives the number of intervals $[\tb,\te]$ with $\tb \in \calHb$, $\te \in \calHf$, $\tb < \te$, and $\te - \tb < \tm$. 
For instances 1 to 4, demand-level shift planning has been done as a preprocessing. All the shifts $\sh$ are on a fixed time interval $[\tb[\sh],\te[\sh]]$ of six hours, and do not contain lunch breaks. 
We introduce them for testing purpose: the only difficulty in these instances comes from delay costs -- without these stochastic costs, they can be solved in polynomial time using a flow approach.
 On the contrary, the last four instances deal with jobs distributed over the whole day, from 00:29 am to 11h59 pm. 
For these instances, the starting time $\hb$, the ending time $\he$, and the breaks of the shifts must be chosen.
 Finally, the last column gives the maximal number of jobs in a well-scheduled shift.
 \begin {table}[h]
\begin{center}
\begin{tabular}{|c|c|c|c|} 
\hline
Instance & Number of jobs & Number of feasible & Maximal number of \\ 
& &  time intervals & jobs in a shift \\
\hline
1 & 57 & 1 & 4\\
2 & 166 & 1 & 5\\
3 & 232 & 1 & 6\\
4 & 300 & 1 & 7\\
5 & 49 & 251 & 7 \\
6 & 111 & 251 & 7\\
7 & 210 & 251 & 7\\
8 & 256 & 251 & 7\\
\hline

\end {tabular}
\caption {Air France instances \label{tab:AFInstances}}
\end{center}
\end {table}

\subsubsection{Compact MILP approach results.}
\label{subsubsec:CompactMIPresults}
Table \ref{tab:AFResultsTime} provides the computing time required to solve MILP~\eqref{eq:PLNECompact} associated to each instance with respectively 100 scenarios and 200 scenarios.
Due to a time limit of 3600s, the optimality gap reached in the time allocated is also presented for each problem.
In columns ``Optimality gap'', opt means that the instance has been solved to optimality, and symbol $"-"$ means that problem is not solvable due to memory issues, or that no solution has been found in 3600s.
The second and third column provide results for the deterministic version of the problem.
By deterministic we refer to the classical jobs assignment problem, without taking into account any delays nor back-up agents, and use MILP~\eqref{eq:PLNECompact} without variables $y_{v,\oscenario}$ and constraint \eqref{cst:delayArc}.
The fourth and fifth columns, and sixth and seventh columns provide results for the stochastic ground staff shift planning problem as presented in \eqref{pb}, with respectively 100 scenarios and 200 scenarios.\\
Deterministic instances are solved to optimality in a few seconds.
Even for instances with a small number of jobs, MILP \eqref{eq:PLNECompact} is hardly tractable by Gurobi, and the difficulty greatly increases with both the number of jobs and the number of scenarios. In comparison, the deterministic problem is easily solvable by a compact MILP approach. These results confirm that taking into account delays, as stated in \eqref{pb}, makes the problem very hard to solve and legitimate the column generation approach.

\begin{table}[h]
\begin{center}
\begin{tabular}{c|c|c|c|c|c|c|}
\cline{2-7} & \multicolumn{2}{c|}{Deterministic problem} & \multicolumn{2}{c|}{Stochastic problem} &  \multicolumn{2}{c|}{Stochastic problem} \\ 
 & \multicolumn{2}{c|}{} & \multicolumn{2}{c|}{100 scenarios} &  \multicolumn{2}{c|}{200 scenarios} \\ \hline
\multicolumn{1}{|c|}{Instance} & Running time  & Optimality & Running time & Optimality & Running time & Optimality \\ 
\multicolumn{1}{|c|}{} & (hh:mm:ss)  & gap (\%) & (hh:mm:ss)  &  gap (\%) & (hh:mm:ss)  &  gap (\%) \\ \hline
\multicolumn{1}{|c|}{1} & 00:00:01 & opt & 01:00:00 & 7.1 & 01:00:00 & 11.4 \\
\multicolumn{1}{|c|}{2} & 00:00:01 & opt & 01:00:00 & -   & 01:00:00 & - \\ 
\multicolumn{1}{|c|}{3} & 00:00:01 & opt & 01:00:00 & -   & 01:00:00 & - \\
\multicolumn{1}{|c|}{4} & 00:00:02 & opt & 01:00:00 & -   & 01:00:00 & - \\
\multicolumn{1}{|c|}{5} & 00:00:01 & opt & 01:00:00 & 6.7 & 01:00:00 & 8.2  \\
\multicolumn{1}{|c|}{6} & 00:00:01 & opt & 01:00:00 & 6.9 & 01:00:00 & 23.1 \\
\multicolumn{1}{|c|}{7} & 00:00:06 & opt & 01:00:00 & -   & 01:00:00 & -  \\
\multicolumn{1}{|c|}{8} & 00:00:14 & opt & 01:00:00& -    & 01:00:00 & -  \\
\hline
\end{tabular}
\caption {Running time needed to solve the different integer programs and optimality gap reached\label{tab:AFResultsTime}}
\end{center}
\end{table}

\section{Column generation approach \label{sec:CGApproach}}
\paragraph{}

MILP~\eqref{eq:PLNECompact} has a poor relaxation that impedes its resolution.
In order to deal with larger instances, we now introduce a set-partitioning formulation of the stochastic ground staff shift planning problem.
Consider the following \emph{master problem}.
\begin{subequations}\label{eq:MasterProblem}
	\begin{alignat}{2}
	\min_y \enskip&\sum_{\shift \in \calSH}\costshift y_{\shift}
	&\quad &
	\\
	\mathrm{s.t.}\enskip
	& \sum_{\shift \ni \task}y_{\shift} = 1 , && \forall \task \in \setTasks \label{cst:CovertnessCst}\\
	& y_{\shift} \in \{0,1\}, && \forall \shift \in \calSH \label{eq:masterIntegrity}
	\end{alignat}
\end{subequations}
Binary variable $y_{\shift}$ indicates if a shift $\shift \in \calSH$ belongs to the solution. 
The notation $\shift \ni \task$ means that the job $\task$ is realized during the sequence of $\shift$, thus the constraint \eqref{cst:CovertnessCst} ensures that all jobs in $\setTasks$ are covered. 
Master problem~\eqref{eq:MasterProblem} is therefore immediately equivalent to our problem~\eqref{pb}. Remark that we can replace \eqref{eq:masterIntegrity} by $y_{\sh} \in \bbZ_+$ as for each shift $\sh$, $c_{\sh} \geq 0$, and we can replace $=$ by $\geq$ in \eqref{cst:CovertnessCst} because removing a job $j$ from a well-scheduled shift $\sh$ gives a well-scheduled shift $\sh'$ such that $c_{\sh'} \leq c_{\sh}$, 

The advantage of \eqref{eq:MasterProblem} over the compact formulation \eqref{eq:PLNECompact} comes from its stronger relaxation. Indeed consider the Dantzig-Wolfe reformulation of \eqref{eq:PLNECompact} obtained by replacing the polyhedron defined by \eqref{cst:flow} and \eqref{cst:delayArc} by the convex hull of its integer points.
As the vertices of the polyhedron defined by \eqref{cst:flow} correspond to the $o$-$d$ paths in $D$, Proposition~\ref{prop:bijectionDigraph} ensures that this Dantzig-Wolfe reformulation is our master problem~\eqref{eq:MasterProblem}. 
The number of shifts $\calSH$ being exponentially large, \eqref{eq:MasterProblem} cannot be given directly to a solver and must be solved by column generation.



\subsection{Column generation algorithm for the linear relaxation}
\label{sub:column_generation_algorithm_for_the_linear_relaxation}

Our algorithm to solve~\eqref{eq:MasterProblem} requires a subroutine to solve the linear relaxation of~\eqref{eq:MasterProblem}. 
We solve this linear relaxation using Algorithm~\ref{alg:CGalgorithm}, which is a standard column generation algorithm.
In Algorithm~\ref{alg:CGalgorithm}, we denote by $\lambda_j$ the dual variable associated to constraint~\ref{cst:CovertnessCst} for job $J$.
A set of feasible shifts $\calSH' \subset \calSH$ covering all the jobs in $J$ is easily obtained by taking a one job shift for each job in $J$.
Section~\ref{sec:PrincingSubProblem} provides our algorithm to solve the pricing subproblem of Step~\ref{alg:SolvePricingSubProblem}.
\begin{algorithm}
\begin{algorithmic}[1]
\STATE \textbf{input:} a set of jobs $J$, a set of feasible shifts $\calSH' \subset \calSH$ covering all the jobs in $J$, a negative real number $\Delta$;
\REPEAT 
\STATE solve \eqref{eq:MasterProblem} restricted to 
shifts in $\calSH'$ with any standard LP solver; \label{alg:SolveRestrictedDualPb}
\STATE denote $\clow$ its optimal value, and $\lambda_j$ the dual value associated to \eqref{cst:CovertnessCst};
\STATE find a shift $\shift \in \calSH$ whose reduced cost $\tilde{c}_\shift = \costshift - \sum_{\task \in \shift} \lambda_{\task}$ is either less than M or minimal\label{alg:SolvePricingSubProblem}; \\ \hfill (\emph{pricing subproblem})
\IF{$\tilde{c}_\shift > \Delta$} \label{step:updateM}
\STATE $ \Delta \leftarrow \Delta/2$;
\ENDIF
\STATE $\calS' \leftarrow \calS' \cup \{ \shift \}$ \label{alg:addShift}; 
\UNTIL $\tilde{c}_\sh \geq 0 $ \label{alg:endWhileCoutReduit}
\STATE \textbf{return} $\clow$, $\lambda_j$, and $\calS'$;
\end{algorithmic}
\caption{Column Generation algorithm}
\label{alg:CGalgorithm}
\end{algorithm}

Column generation theory ensures that Algorithm~\ref{alg:CGalgorithm} converges after a finite number of iterations, the value of $\clow$ returned by the algorithm is the value of an optimal solution of the linear relaxation of~\eqref{eq:MasterProblem}, 
and $\lambda_j$ is the dual value associated to $y_j$ in a basic optimal solution of the relaxation. 
Indeed, simplex algorithm theory ensures that an optimality criterion of a solution of the linear relaxation of~\eqref{eq:MasterProblem} is the non-negativity of the corresponding \emph{reduced costs}
\begin{equation}\label{eq:optimalitySimplex} 	
\tilde{c}_\shift = \costshift - \sum_{\task \in \shift} \lambda_{\task} \geq 0 \text{ for each $\sh$ in $\calSH$.}
 \end{equation} 
Algorithm~\ref{alg:CGalgorithm} considers a master problem \eqref{eq:MasterProblem} restricted to a set of shifts $\calSH'$ of tractable size, and dynamically adds shifts to $\calSH'$ until the optimality criterion \eqref{eq:optimalitySimplex} is matched. 
Additional details on column generation can be found in \citep{lubbecke2010column}. 

\paragraph{}
The only originality of Algorithm~\ref{alg:CGalgorithm} lies in the use of the constant $\Delta$.
In the usual column generation, the pricing subproblem seeks a column $\sh$ of minimum reduced cost $\tilde{c}_\shift = \costshift - \sum_{\task \in \shift} \lambda_{\task}$.
Solving the pricing subproblem is generally the time consuming part of a column generation.
It is well known that adding an arbitrary shift of negative reduced cost instead of a shift of minimal reduced cost is sufficient to ensure convergence to an optimal solution, and leads to an easier pricing subproblem.
However, adding such an arbitrary shift instead of the minimal one strongly increases the number of iterations on our problem.
The use of the constant $\Delta$ enables to turn the pricing subproblem into ``find a column of nearly minimal reduced cost'', and have the advantages of both previous options: it leads to an easier pricing subproblem without increasing to much the number of iterations of the column generation.



\subsection{Exact column generation scheme}
\label{sub:exact_column_generation_scheme}
 
We can now state Algorithm~\ref{alg:RetrieveIntegerSolution}, our exact solution algorithm for problem~\eqref{eq:MasterProblem}.
\begin{algorithm}
\begin{algorithmic}[1]
\STATE \textbf{Input: } a set of jobs $J$, a set of feasible shifts $\calSH'' \subset \calSH$ covering all the jobs in $J$;
\STATE solve the linear relaxation of~\eqref{eq:MasterProblem} using Algorithm~\ref{alg:CGalgorithm} with $\calSH' = \calSH''$. Denote $\clow$, $\lambda_j$, and $\calSH'$ the values returned. \label{step:relax} 
\STATE solve \eqref{eq:PLNECompact} restricted to $\calSH'$ using any standard MILP solver \label{enum:UpperBound}; denote $\cupp$ its optimal solution;
\IF {$\cupp = \clow$}
\STATE \textbf{stop}; \textbf{return} $\cupp$; \label{step:noGap}
\ENDIF 
\STATE add to $\calS'$ all the shifts $\shift$ in $\calSH$ satisfying $\tilde{c}_\shift = \costshift - \sum_{\task \in \shift} \lambda_{\task} \leq \cupp - \clow$;\label{step:addAllSmallerGap}
\STATE solve \eqref{eq:MasterProblem} restricted to $\calS'$ \label{enum:IntegerSolution} with any MILP standard solver; 
\STATE denote $\cfinal$ its optimal solution; \label{step:cfinal}
\STATE \textbf{return} $\cfinal$;
\end{algorithmic}
\caption{Exact solution algorithm for~\eqref{eq:MasterProblem}}
\label{alg:RetrieveIntegerSolution}
\end{algorithm}
Lemma~\ref{lem:reducedCostGap} below ensures that the value returned is the optimal value of \eqref{eq:MasterProblem}.
Indeed, at Step~\ref{step:addAllSmallerGap}, $\cupp$ is the value of a solution of~\eqref{eq:MasterProblem} and hence an upper bound on its optimal solution, and the previous section ensures that $\clow$ and $(\tilde{c}_\sh)_{\sh \in \calSH}$ are respectively the value and reduced costs of an optimal solution of the linear relaxation of \eqref{eq:MasterProblem}.

\begin{lemma}\label{lem:reducedCostGap}\citep[Proposition 2.1, p. 389]{nemhauser1988integer}
Consider an integer program in standard form with variables $(x_i)$ for which the linear relaxation admits a finite optimal value $\bar{v}$. Suppose given an upper bound \emph{UB} on the optimal value of the integer program. Then for every i such that $\tilde{c}_{i}> UB - \bar{v}$, the variable $x_{i}$ is equal to 0 in all optimal solutions of the integer program where $\tilde{c}_{i}$ denotes the reduced cost of the variable $z_{i}$ when the linear relaxation has been solved to optimality.
\end{lemma}

Again, a set of feasible shifts $\calSH'' \subset \calSH$ covering all the jobs in $J$ is easily obtained by taking a one job shift for each job in $J$.
A potential limit of Algorithm~\ref{alg:RetrieveIntegerSolution} is that the number of pairing generated at Step~\ref{step:addAllSmallerGap} might be intractably large.
In that case, Algorithm~\ref{alg:RetrieveIntegerSolution} should be replaced by a branch and price.
However, in all our numerical experiments, $\cupp - \clow$ happens to be very small, the number of pairing generated at Step~\ref{step:addAllSmallerGap} remains tractable, and Algorithm~\ref{alg:RetrieveIntegerSolution} enables to solve the problem to optimality.

\subsection{Numerical results}

\paragraph{}
We run the experiments under the settings of Section~\ref{subsubsec:ExperimentalSettings} and consider the instances introduced in Section~\ref{subsubsec:AFInstances} with the same scenarios.

\paragraph{}
Table \ref{tab:AFTimeCG100Scnearios} and Table~\ref{tab:AFTimeCG200Scnearios} present the performances of Algorithm~\ref{alg:RetrieveIntegerSolution} on our instances of the stochastic ground staff shift planning problem~\eqref{pb}, with respectively 100 scenarios and 200 scenarios.
In both tables, the first column gives the name of the instance solved by Algorithm~\ref{alg:RetrieveIntegerSolution}.
\emph{All instances are solved to optimality}.
The second column gives the number of steps in the column generation Algorithm~\ref{alg:CGalgorithm}.
The third, fourth and fifth columns provide respectively the percentage of time spent in the pricing subproblem, in the column generation, and in retrieving the integer solution, i.e., in the steps \ref{enum:UpperBound} to \ref{step:cfinal} of Algorithm~\ref{alg:RetrieveIntegerSolution}. 
The sixth column indicates the total time needed to solve problem~\eqref{eq:MasterProblem} using Algorithm~\ref{alg:RetrieveIntegerSolution}.
Finally, the last column presents the percentage by which the cost of the solution returned by Algorithm~\ref{alg:RetrieveIntegerSolution} is smaller than the cost of the solution obtained using the deterministic approach that does not take into account delay -- the optimal solution of the deterministic problem is obtained using the MIP of Section~\ref{subsubsec:CompactMIPresults}.\\

\begin {table}[h]
\begin{center}
\begin{tabular}{|c|c|c|c|c|c|c|}
\hline
Instance&Number&Pricing &Col. Gen.&Integer solution&Total time&Improve.~vs\\
 &of steps& time (\%) & time (\%) & time (\%) &(hh:mm:ss)& det.~prob. (\%)\\
\hline

1  & 7 & 99.2  & 99.3 & 0.3 & 00:00:02 & 3.12\\
2  & 24 & 94.8  & 94.8 & 5.2 & 00:01:41 & 4.34 \\
3  & 34 & 99.2  & 99.2 & 0.5 & 00:12:34 & 4.95 \\
4  & 56 & 93.8  & 93.9 & 6.1 & 03:35:23 & 4.83 \\
5  & 13 & 90.1  & 94.1 & 2.1 & 00:00:11 & 3.12 \\
6  & 14 & 96.3  & 96.9 & 3.1 & 00:02:01 & 4.24\\
7  & 22 & 99.4  & 99.5 & 0.2 & 00:23:20 & 5.60 \\
8  & 36 & 95.6  & 95.7 & 4.3 & 01:35:59 & 5.36\\
\hline
\end {tabular}
\caption {Performances of Algorithm~\ref{alg:RetrieveIntegerSolution} to solve master problem \eqref{eq:MasterProblem} with 100 scenarios to optimality
\label{tab:AFTimeCG100Scnearios}}
\end{center}
\end {table}

\begin {table}[h]
\begin{center}
\begin{tabular}{|c|c|c|c|c|c|c|}
\hline
Instance&Number&Pricing &Col. Gen.&Integer solution&Total time&Improve.~vs\\
 &of steps& time (\%) & time (\%) & time (\%) &(hh:mm:ss)& det.~prob. (\%)\\
\hline

1  & 5 & 98.6 & 99.1 & 0.09 & 00:00:02 & 3.00 \\
2  & 21 & 98.5 & 98.9 & 0.7 & 00:02:20 & 4.02 \\
3  & 47 & 93.1 & 93.2 & 6.8 & 00:19:59 & 4.83 \\
4  & 115 & 89.1 & 89.1 & 10.1 & 05:44:41 & 4.62 \\
5  & 8 & 85.1 & 87.2 & 11.8 & 00:00:22 & 3.31  \\
6  & 16 & 98.4 & 98.5 & 0.1 & 00:02:10 & 4.32  \\
7  & 25 & 99.4 & 99.5 & 0.2 & 00:27:22 & 5.49  \\
8  & 38 & 99.7 & 99.8 & 0.1 & 03:43:09 & 5.17  \\
\hline
\end {tabular}
\caption {Performances of Algorithm~\ref{alg:RetrieveIntegerSolution} to solve master problem \eqref{eq:MasterProblem} with 200 scenarios to optimality
\label{tab:AFTimeCG200Scnearios}}
\end{center}
\end {table}

We underline that these results show the interest of using a stochastic ground staff shift planing approach that includes the cost of delay instead of a deterministic one which does not take into account delay. 
Indeed, this enables to reduce the total operating cost by 3\% to 5\%.
This improvement seems to increase with the number of jobs, probably due to additional flexibility in the design job sequences.

The column generation approach enables to solve to optimality instances with up to 250 scenarios in a few minutes, and instances with up to 300 scenarios in a few hours.
The number of scenarios does not impact too much the computing time.
The critical element in the performance is the pricing subproblem: most of the computing time is spent solving instances of the pricing subproblem.
And the factor that limits the size of the instance we can solve is the pricing subproblem algorithm -- Algorithm~\ref{alg:enumeration} introduced in the next section.
Indeed, on larger instances, the list $\sfL$ of candidate paths used by Algorithm~\ref{alg:enumeration} becomes too large and saturates the memory of our computer.





Finally, Tables \ref{tab:AFResultsComparaison100Scenarios} and \ref{tab:AFResultsComparaison200Scenarios} provide results explaining why we can use Algorithm~\ref{alg:RetrieveIntegerSolution} instead of a branch-and-price,
for instances with respectively 100 and 200 scenarios.
In both tables, the first column indicates the instance solved.
The column ``final integrality gap'' provides the gap $(c_r - \clow)/ \clow$ between the optimal integer value $c_r$ of \eqref{eq:MasterProblem} and the value $\clow$ of its linear relaxation computed at Step~\ref{step:relax} of Algorithm~\ref{alg:RetrieveIntegerSolution}.
The column ``intermediate integrality gap'' provides the gap $(\cupp - \clow)/ \clow$ between the linear relaxation $\clow$ of \eqref{eq:MasterProblem} and the value $\cupp$ of the integer solution computed at Step~\ref{enum:IntegerSolution}  of Algorithm~\ref{alg:RetrieveIntegerSolution}.
The fourth column gives the number of variables satisfying the condition of step~\ref{step:addAllSmallerGap} of Algorithm~\ref{alg:RetrieveIntegerSolution}, thus including both the ones generated through the column generation in Algorithm~\ref{alg:CGalgorithm} and the ones generated through step \ref{step:addAllSmallerGap} in Algorithm~\ref{alg:RetrieveIntegerSolution}.
We underline that the integrality gap at the end of the column generation, $(\cupp - \clow)/ \clow$, is very small.
This explains the practical efficiency of Algorithm~\ref{alg:RetrieveIntegerSolution}.
Indeed, for 10 instances out of 16, there is no integrality gap, and Algorithm~\ref{alg:RetrieveIntegerSolution} stops at Step~\ref{step:noGap}.
On the other instances, the integrality gap is non-greater than $0.04\%$ and at most 18345 columns are generated at Step~\ref{step:addAllSmallerGap}. Furthermore, in these cases, the optimal integer solution is found at step \ref{enum:UpperBound} of Algorithm~\ref{alg:RetrieveIntegerSolution}, and the next steps only prove its optimality.


\begin {table}[h]
\begin{center}
\begin{tabular}{|c|c|c|c|}
\hline
Instance & Final integrality & Intermediate integrality & Number of variables satisfying \\ 
 & gap (\%) & gap (\%) & $\tilde{c}_\shift \leqslant (\cupp - \clow)$\\
\hline

1  & 0.0 & 0.0& --\\
2  & 0.025 & 0.025 & 4339\\
3  & 0.0 & 0.0 & --\\
4  & 0.026 & 0.026 & 11761\\
5  & 0.0 & 0.0 & --\\
6  & 0.027 & 0.027 & 1501\\
7  & 0.0 & 0.0 & --\\
8  & 0.0 & 0.0 & --\\
\hline
\end {tabular}
\caption {Numerical results on master problem~\eqref{eq:MasterProblem} with 100 scenarios. All instances are solved to optimality. \label{tab:AFResultsComparaison100Scenarios}}
\end{center}
\end {table}

\begin {table}[h]
\begin{center}
\begin{tabular}{|c|c|c|c|}
\hline
Instance & Final integrality & Intermediate integrality & Number of variables satisfying \\ 
 & gap (\%) & gap (\%) & $\tilde{c}_\shift \leqslant (\cupp - \clow)$\\
\hline


1  & 0.0 & 0.0 & --\\
2  & 0.0 & 0.0 & --\\
3  & 0.006 & 0.006 & 5410\\
4  & 0.036 & 0.036 & 18345\\
5  & 0.03 & 0.03 & 321 \\
6  & 0.0 & 0.0 & --\\
7  & 0.0 & 0.0 & --\\
8  & 0.0 & 0.0 & --\\
\hline
\end {tabular}
\caption {Numerical results on master problem~\eqref{eq:MasterProblem} with 200 scenarios. All instances are solved to optimality. \label{tab:AFResultsComparaison200Scenarios}}
\end{center}
\end {table}

\section{Pricing subproblem}
\label{sec:PrincingSubProblem}
\paragraph{}
We now introduce a solution scheme for the pricing problem 
\begin{equation}\label{eq:pricingSubproblem}
    \min_{\shift \in \calSH}\costshift - \sum_{\task \in \shift} \lambda_{\task} 
\end{equation}
and its variants solved at Step~\ref{alg:SolvePricingSubProblem} of Algorithm~\ref{alg:CGalgorithm} and Step~\ref{step:addAllSmallerGap} of Algorithm~\ref{alg:RetrieveIntegerSolution}.
The difficulty in the pricing subproblem lies in the non-linearity and the stochasticity of shift costs defined in Equation~\eqref{eq:shiftCost}.
To overcome this difficulty, we model the pricing subproblem within the framework for stochastic resource constrained shortest path problems recently introduced by the second author \citep{parmentier2015algorithms}.

\subsection{Framework and algorithm \label{subsec:FrameworkAndAlgorithm}}



We follow \citet{2017arXiv170606901P} in their presentation of the \MRCSP  framework \citep{parmentier2015algorithms}.
\subsubsection{Algebraic framework.}
A binary operation $\plus$ on a set $M$ is \emph{associative} if $q \plus (q' \plus q'') = (q \plus q')\plus q''$ for $q,q',$ and $q''$ in $M$. An element $e$ is \emph{neutral} if $e \plus q =  q \plus e = q$ for any $q$ in $M$. 
A pair $(M,\plus)$, where $M$ is a set and $\plus$ a binary operator on $M$, is a \emph{monoid} if $\plus$ is associative and admits a neutral element.
A partial order $\moins$ is \emph{compatible} with $\plus$ if for $q,q'$, and $q''$ in $M$, $q\moins q'$ implies $q \plus q'' \moins q' \plus q''$ and $ q'' \plus q \moins q'' \plus q'$. A partially ordered set $(M,\moins)$ is a \emph{lattice} if any pair $(q,q')$ of elements of $M$ admits a greatest lower bound or \emph{meet} denoted by $q\meet q'$, and a least upper bound or \emph{join} denoted by $q \join q'$.
\begin{definition}
A set $(M,\plus,\moins)$ is a \emph{lattice ordered monoid} if $(M,\plus)$ is a monoid, $(M,\moins)$ is a lattice, and $\moins$ is compatible with $\plus$.
\end{definition} 
Given a digraph $D=(V,A)$, a lattice ordered monoid $(M,\plus,\moins)$, elements $q_a\in M$ for each $a\in A$, origin and destination vertices $o$ and $d$, and two non-decreasing mappings $c : M \rightarrow \bbR$ and $\rho:M \rightarrow \{0,1\}$, the \MRCSP seeks
$$\text{an $o$-$d$ path $P$ of minimum $c\left(\bigoplus_{a\in P}q_{a}\right)$ among those satisfying $\rho\left(\bigoplus_{a\in P}q_{a}\right) = 0$.} $$ 
We call $q_a$ the {\em resource} of the arc $a$.
The sum $\bigoplus_{a\in P}q_{a}$ is the \emph{resource} of a path $P$, and we denote it by $q_{P}$. The real number $c\left(q_P\right)$ is its \emph{cost}, and the path $P$ is \emph{feasible} if $\rho\left(q_P\right)$ is equal to $0$. We therefore call $c$ and $\rho$ the {\em cost} and the {\em infeasibility functions}. 

\subsubsection{Enumeration algorithm.}\label{subsubsec:Enumeration algorithm}
We now describe an \emph{enumeration algorithm} for the \MRCSP.
The algorithm takes in input a collection $(B_v)_{v \in V}$ of sets $B_v$ of lower bounds $b$ in $M$ such that, for each $v$-$d$ path $Q$, there exists a bound $b$ in $B_v$ satisfying $b \leq q_Q$.
We explain later how such bounds are computed.
We define $\key(P)$ as 
\begin{equation}\label{eq:keyDefinition}
    \key(P) = \min \{c(q_P \plus b)\colon b \in B_v, \rho(q_P \plus b)=0\} \quad \text{where $v$ is the last vertex of $P$,}
\end{equation}
and the minimum is equal to $+\infty$ if the set is empty.
The empty path at a vertex $v$ is the path with no arcs starting and ending at vertex $v$. Its resource is the neutral element of the monoid. We denote by $P+a$ the path composed of a path $P$ followed by an arc $a$.
The algorithm maintains a list $\mathsf{L}$ of partial paths, an upper bound $c_{od}^{UB}$ on the cost of an optimal solution, and
lists $(\mathsf{L}_{v}^{\mathrm{nd}})_{v \in V}$ of non-dominated paths ending with vertex $v$
are maintained. 
Algorithm~\ref{alg:enumeration} states our algorithm. 

\begin{algorithm}
\begin{algorithmic}[1]
\STATE \textbf{input}: A graph $D = (V,A)$, resource $(q_a) \in M^A$, cost and infeasibility functions $c$ and $\rho$, and sets of bounds $(B_v)_{v \in V}$, ;
\STATE \textbf{initialization:}  $c_{od}^{UB} \leftarrow +\infty$ and $\mathsf{L} \leftarrow \emptyset$ 
\STATE add the empty path at the origin $o$ to $\mathsf{L}$
\WHILE{$\mathsf{L}$ is not empty}
\STATE Extract from $\sfL$ a path $P$ of minimum $\key(P)$ in $\mathsf{L}$; \label{step:pathSelection} 
\STATE $v\leftarrow$ last vertex of $P$;
\IF{$v=D$, $\rho(q_P) = 0$, and $c(q_P)<c_{od}^{UB}$} \label{step:extStart}
\STATE $c_{od}^{UB} \leftarrow c(q_{P}) $; \label{step:codub}
\ELSE
\FORALL{$a\in \delta^{+}(v)$ \textbf{such that} $\key(P+a) < c_{od}^{UB}$} \label{step:LBtest}
\STATE $\mathsf{L} \leftarrow \mathsf{L} \cup \{P+a\}$
\ENDFOR
\ENDIF \label{step:extEnd}
\ENDWHILE
\STATE \textbf{return} $c_{od}^{UB}$;
\end{algorithmic}
\caption{Enumeration algorithm for the \MRCSP}
\label{alg:enumeration}
\end{algorithm}

The rationale behind this algorithm is the following: given and $o$-$v$ path $P$, $\key(P)$ is a lower bound on the cost of any feasible $o$-$d$ path starting by $P$.
Indeed, for any $v$-$d$ path $Q$, there is a $b$ in $B_v$ such that $b \preceq q_Q$, and $\key(P) \leq c(q_P \oplus b) \leq c(q_P \oplus q_Q) = c(q_{P+Q})$.
And $c_{od}^{UB}$ is the cost of the best feasible $o$-$d$ path found. 
Hence, if $\key(P) \geq c_{od}^{UB}$, there is no $o$-$d$ path starting by $P$ that is better than the best path found. The algorithm only enumerates all the paths satisfying $\key(P) < c_{od}^{UB}$.

\begin{proposition}\label{prop:MRCSPalgoConvergence}\emph{\textbf{\citet{2017arXiv170606901P}}}\\
Suppose that $D$ is acyclic. Then Algorithm~\ref{alg:enumeration} converges after a finite number of iterations, and, at the end of the algorithm, $c_{od}^{UB}$ is equal to the cost of an optimal solution of the \MRCSP\enskip if such a solution exists, and to $+\infty$ otherwise. 
\end{proposition}

This algorithm differs from the resource constrained shortest path algorithms in the literature \citep{irnich2005shortest} by the fact that it relies on bounds to discard paths, while the standard algorithms rely on dominance. 
On stochastic resource constrained shortest path problems, dominance between paths is rare and does not enable to discard paths efficiently \citep{parmentier2015algorithms}.
Algorithm~\ref{alg:enumeration} is in fact a generalization to resource constrained shortest path problem of the A$^*$ algorithm \citep{bast2016route}.
One advantage of the \MRCSP framework \citep{parmentier2015algorithms} is that it gives an algorithm which,
 given an instance of the \MRCSP and an integer $\kappa$,
 computes sets of bounds $B_v$ of size $\kappa$ that can be used by Algorithm~\ref{alg:enumeration}.
 In our numerical experiments, we use this algorithm with $\kappa = 100$, following the advices of \citet[Section 4.3]{2017arXiv170606901P} on the choice of $\kappa$. 

\begin{remark}
Step~\ref{eq:pricingSubproblem} of Algorithm~\ref{alg:CGalgorithm} requires to solve the following variant of the \MRCSP.
\begin{center}
\emph{Find an $o$-$d$ path $P$ satisfying $c(q_P) \leq \Delta$ or an $o$-$d$ path $Q$ of minimum $c(q_Q)$.}
\end{center}
\noindent
Algorithm~\ref{alg:enumeration} is easily adapted to this variant. Indeed, it suffices to stop the algorithm and return the path $P$ at Step~\ref{step:codub} if $c(q_P) \leq \Delta$. Step~\ref{step:addAllSmallerGap} of Algorithm~\ref{alg:RetrieveIntegerSolution} requires to solve the following variant
\begin{center}
\emph{Generate all the $o$-$d$ paths $P$ satisfying $c(q_P) \leq \cupp - \clow$.}
\end{center}
\noindent Again, we can easily adapt Algorithm~\ref{alg:enumeration}. 
It suffices to maintain a set $\mathsf{S}$ of solutions (initially empty), to replace $c_{od}^{UB}$ by $c^{\mathrm{upp}} - c^{\mathrm{low}}$ in Steps~\ref{step:extStart} and~\ref{step:LBtest}, 
to replace Step~\ref{step:codub} by $\mathsf{S} \leftarrow \mathsf{S} \cup \{P\}$, and to return $\mathsf{S}$. 
The set $\mathsf{S}$ returned contains all the $o$-$d$ paths $P$ satisfying 
$\rho(q_{P}) = 0$ and $c(q_{P}) \leq c^{\mathrm{upp}} - c^{\mathrm{low}}$.
\end{remark}

\subsection{Modeling the pricing sub-problem}
We now explain how to model our pricing sub-problem \eqref{eq:pricingSubproblem} as a \MRCSP.

\subsubsection{Digraph.} 
\label{ssub:digraph}

We use the digraph $D$ defined in Section~\ref{sec:compactMIP}. As Proposition~\ref{prop:bijectionDigraph} ensures that there is a bijection between $o$-$d$ paths in $D$ and well scheduled shifts, we only need to define a lattice ordered monoid to take into account shifts costs.


\subsubsection{A lattice ordered monoid modeling job successions.} 
We now introduce a monoid that enables to count the number of rescheduled jobs in the shift corresponding to an $o$-$d$ path.
Consider an $o$-$d$ path $P$ composed of an $o$-$v$ path $Q$ followed by a $v$-$d$ path $R$.
Then the jobs in $R$ that will be rescheduled in the shift of $P$ depends on $Q$. 
Indeed, consider the red path $P$  on Figure~\ref{fig:ModelingMIP}, that we partition into $Q$ composed of $o$, $(j_2,\hb[2],\bl)$, and $(j_3,\hb[2],\al)$, and $R$ composed of $(j_3,\hb[2],\bl)$, $\he[1]$, and $d$.
Suppose that $\xe[j_2] + \tbr > \xb[j_3]$.
Then within the shift encoded by the red path, $j_3$ is rescheduled. 
But within the shift encoded by the path $P'$ obtained by replacing $Q$ by $o$, $(j_3,\hb[2],\al)$, job $j_3$ is not rescheduled. 
Hence, the resource of the $v$-$d$ path $R$ must contain enough information to identify the number of jobs of $R$ that will be rescheduled in any $o$-$d$ path ending by $R$.
But remark that, given a subpath $R$ of an $o$-$d$ path $P$, knowing if the first job of $R$ is rescheduled in the shift $\sh$ of $P$, rules~\ref{case:succ} to~\ref{case:late} enable to identify which jobs of $R$ are rescheduled in $\sh$.
Figure~\ref{fig:ExModelResource} illustrates jobs corresponding to such subpaths.
Hence, in all the $o$-$d$ paths containing $R$, there are only two possible options on the jobs of $R$ that will be rescheduled: the jobs rescheduled when the first job is rescheduled, and the jobs rescheduled when it is not.
We now build our monoid based on this observation.








Let $\sbg$, $\sdo$, and $\sdt$ be respectively identified with $\bbZ_+$, $\bbZ_+^2$, and $\bbZ_+^2$. Elements of $\sbg$, $\sdo$, and $\sdt$ are respectively denoted $\rbg$, $\rdo$, and $\rdt$, where $\tbg$, $\cdo$, $\tdo$, $\cdt$, and $\tdt$ belong to $\bbZ_+$.
The monoid $\monoidS$ we use to model the number of jobs realized by back-up agents is of the form $\monoidS = \{e,\infty \} \cup S $ where $S$ is the following subset of $\sbg \times \sdo \times \sdt$ 
\begin{equation}\label{eq:definitionMonoidSet}
    S = \left\{\res 
    \left|
    \begin{array}{ll}
    \heurefinUn < \heurefinNun \enskip &  \Rightarrow \coutUn = \coutNun \\
    \heurefinUn >\heurefinNun  \enskip &  \Rightarrow \coutUn = \coutNun - 1 \\
    \heurefinUn = \heurefinNun \enskip &  \Rightarrow (\coutUn = \coutNun \text{ or } \coutUn = \coutNun -1)
    \end{array}
    \right.\right\}
\end{equation}
The semantic of our monoid is as follows.
Each path $P$ in $D$ is decorated by a vector of resources in $\monoidS$, one for each scenario $\omega$.
Element $e$ is the neutral element and decorates paths $P$ containing no arcs. 
Element $\infty$ is the supremum. 
Suppose now that a path $P$ is decorated with $\res$ for $\omega$, and let $\sh$ be the corresponding partial shift.
Then $\rbg$ indicates the time $\tbg  = \xb[j](\omega)$ at which the first job $j$ of $\sh$ starts under $\omega$, being it rescheduled or not. 
Component $\rdo$ (resp.~$\rdt$) indicates the number $\cdo$ (resp.~$\cdt$) of rescheduled jobs in $\sh$ and the time $\tdo$ (resp~$\tdt$) at which the team operating $\sh$ finishes the last non rescheduled job of $\sh$ under the hypothesis that the first job of $\sh$ is not rescheduled (resp.~rescheduled).
These different quantities are illustrated on Figure~\ref{fig:ExModelResource}.

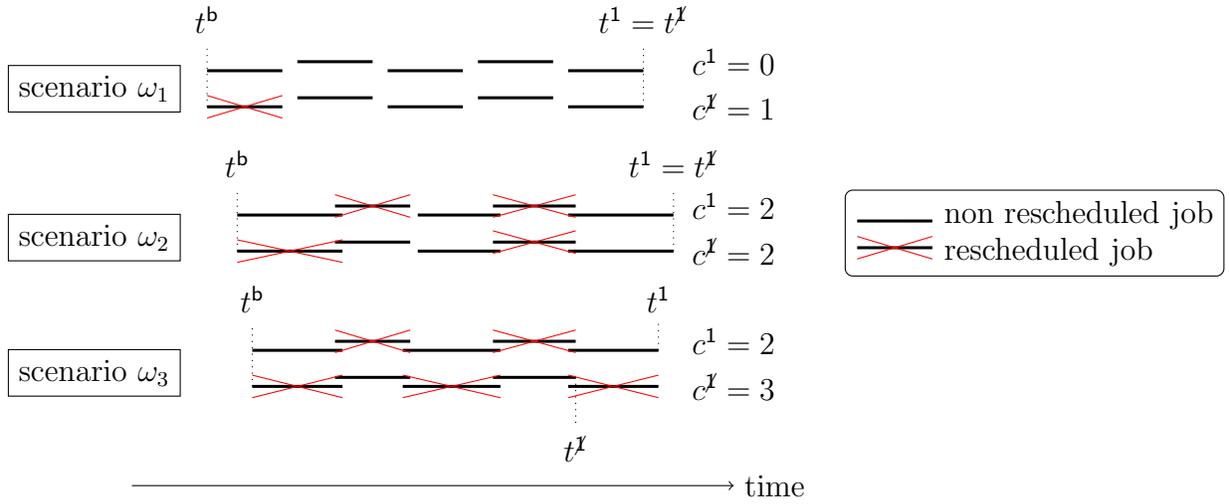
\begin{figure}[h]
    \centering
    \begin{tikzpicture}
            \def\l{1}
            \def\h{0.6}
            \def\xh{.15}
                \draw[->] (-1*\l,-8.2*\h)--(7*\l,-8.2*\h) node[right]{time};

                \draw[line width = 0.4mm] (0*\l,1*\h)--(1*\l,1*\h);
                \draw[line width = 0.4mm] (1.2*\l,1*\h+0.2*\h)--(2.2*\l,1*\h+0.2*\h);
                \draw[line width = 0.4mm] (2.4*\l,1*\h)--(3.4*\l,1*\h);
                \draw[line width = 0.4mm] (3.6*\l,1*\h+0.2*\h)--(4.6*\l,1*\h+0.2*\h);
                \draw[line width = 0.4mm] (4.8*\l,1*\h)--(5.8*\l,1*\h);

                \draw[line width = 0.4mm] (0*\l,1*\h - 0.8*\h)--(1*\l,1*\h - 0.8*\h);
                \draw[line width = 0.4mm] (1.2*\l,1*\h+0.2*\h - 0.8*\h)--(2.2*\l,1*\h+0.2*\h - 0.8*\h);
                \draw[line width = 0.4mm] (2.4*\l,1*\h - 0.8*\h)--(3.4*\l,1*\h - 0.8*\h);
                \draw[line width = 0.4mm] (3.6*\l,1*\h+0.2*\h - 0.8*\h)--(4.6*\l,1*\h+0.2*\h - 0.8*\h);
                \draw[line width = 0.4mm] (4.8*\l,1*\h - 0.8*\h)--(5.8*\l,1*\h - 0.8*\h);
                \draw[red] (0*\l,1*\h - 0.8*\h + \xh)--(1*\l,1*\h - 0.8*\h - \xh);
                \draw[red] (0*\l,1*\h - 0.8*\h - \xh)--(1*\l,1*\h - 0.8*\h + \xh);

                \draw[dotted] (0*\l,1*\h - 0.8*\h) -- ++(0,1.4*\h) node[above]{$\tbg$};
                \draw[dotted] (5.8*\l,1*\h - 0.8*\h) -- ++(0,1.4*\h) node[above]{$\tdo = \tdt$};
                \node at (7*\l,1.2*\h) {$\cdo = 0$};
                \node at (7*\l,0.2*\h) {$\cdt = 1$};

                \draw[line width = 0.4mm] (0.4*\l,-2.2*\h)--(1.8*\l,-2.2*\h);
                \draw[line width = 0.4mm] (1.7*\l,-2.2*\h+0.2*\h)--(2.7*\l,-2.2*\h+0.2*\h);
                \draw[line width = 0.4mm] (2.8*\l,-2.2*\h)--(3.9*\l,-2.2*\h);
                \draw[line width = 0.4mm] (3.8*\l,-2.2*\h+0.2*\h)--(4.9*\l,-2.2*\h+0.2*\h);
                \draw[line width = 0.4mm] (4.8*\l,-2.2*\h)--(6.2*\l,-2.2*\h);
                \draw[red] (1.7*\l,-2.2*\h+0.2*\h+\xh)--(2.7*\l,-2.2*\h+0.2*\h-\xh);
                \draw[red] (1.7*\l,-2.2*\h+0.2*\h-\xh)--(2.7*\l,-2.2*\h+0.2*\h+\xh);
                \draw[red] (3.8*\l,-2.2*\h+0.2*\h+\xh)--(4.9*\l,-2.2*\h+0.2*\h-\xh);
                \draw[red] (3.8*\l,-2.2*\h+0.2*\h-\xh)--(4.9*\l,-2.2*\h+0.2*\h+\xh);

                \draw[line width = 0.4mm] (0.4*\l,-2.2*\h-0.8*\h)--(1.8*\l,-2.2*\h-0.8*\h);
                \draw[line width = 0.4mm] (1.7*\l,-2.2*\h+0.2*\h-0.8*\h)--(2.7*\l,-2.2*\h+0.2*\h-0.8*\h);
                \draw[line width = 0.4mm] (2.8*\l,-2.2*\h-0.8*\h)--(3.9*\l,-2.2*\h-0.8*\h);
                \draw[line width = 0.4mm] (3.8*\l,-2.2*\h+0.2*\h-0.8*\h)--(4.9*\l,-2.2*\h+0.2*\h-0.8*\h);
                \draw[line width = 0.4mm] (4.8*\l,-2.2*\h-0.8*\h)--(6.2*\l,-2.2*\h-0.8*\h);
                \draw[red] (0.4*\l,-2.2*\h+\xh-0.8*\h)--(1.8*\l,-2.2*\h-\xh-0.8*\h);
                \draw[red] (0.4*\l,-2.2*\h-\xh-0.8*\h)--(1.8*\l,-2.2*\h+\xh-0.8*\h);
                \draw[red] (3.8*\l,-2.2*\h-0.8*\h+0.2*\h+\xh)--(4.9*\l,-2.2*\h-0.8*\h+0.2*\h-\xh);
                \draw[red] (3.8*\l,-2.2*\h-0.8*\h+0.2*\h-\xh)--(4.9*\l,-2.2*\h-0.8*\h+0.2*\h+\xh);
                
                \draw[dotted] (0.4*\l,-2.2*\h-0.8*\h) -- ++(0,1.4*\h) node[above]{$\tbg$};
                \draw[dotted] (6.2*\l,-2.2*\h-0.8*\h) -- ++(0,1.4*\h) node[above]{$\tdo = \tdt$};
                \node at (7*\l,1.2*\h-3.2*\h) {$\cdo = 2$};
                \node at (7*\l,0.2*\h-3.2*\h) {$\cdt = 2$};               
                \draw[line width = 0.4mm] (0.6*\l,-5.2*\h)--(1.8*\l,-5.2*\h);
                \draw[line width = 0.4mm] (1.7*\l,-5.2*\h+0.2*\h)--(2.7*\l,-5.2*\h+0.2*\h);
                \draw[line width = 0.4mm] (2.6*\l,-5.2*\h)--(3.9*\l,-5.2*\h);
                \draw[line width = 0.4mm] (3.8*\l,-5.2*\h+0.2*\h)--(4.9*\l,-5.2*\h+0.2*\h);
                \draw[line width = 0.4mm] (4.8*\l,-5.2*\h)--(6.0*\l,-5.2*\h);
                \draw[red] (1.7*\l,-5.2*\h+0.2*\h+\xh)--(2.7*\l,-5.2*\h+0.2*\h-\xh);
                \draw[red] (1.7*\l,-5.2*\h+0.2*\h-\xh)--(2.7*\l,-5.2*\h+0.2*\h+\xh);
                \draw[red] (3.8*\l,-5.2*\h+0.2*\h+\xh)--(4.9*\l,-5.2*\h+0.2*\h-\xh);
                \draw[red] (3.8*\l,-5.2*\h+0.2*\h-\xh)--(4.9*\l,-5.2*\h+0.2*\h+\xh);

                \draw[line width = 0.4mm] (0.6*\l,-5.2*\h-0.8*\h)--(1.8*\l,-5.2*\h-0.8*\h);
                \draw[line width = 0.4mm] (1.7*\l,-5.2*\h+0.2*\h-0.8*\h)--(2.7*\l,-5.2*\h+0.2*\h-0.8*\h);
                \draw[line width = 0.4mm] (2.6*\l,-5.2*\h-0.8*\h)--(3.9*\l,-5.2*\h-0.8*\h);
                \draw[line width = 0.4mm] (3.8*\l,-5.2*\h+0.2*\h-0.8*\h)--(4.9*\l,-5.2*\h+0.2*\h-0.8*\h);
                \draw[line width = 0.4mm] (4.8*\l,-5.2*\h-0.8*\h)--(6.0*\l,-5.2*\h-0.8*\h);
                \draw[red] (0.6*\l,-5.2*\h+\xh-0.8*\h)--(1.8*\l,-5.2*\h-\xh-0.8*\h);
                \draw[red] (0.6*\l,-5.2*\h-\xh-0.8*\h)--(1.8*\l,-5.2*\h+\xh-0.8*\h);
                \draw[red] (2.6*\l,-5.2*\h+\xh-0.8*\h)--(3.9*\l,-5.2*\h-\xh-0.8*\h);
                \draw[red] (2.6*\l,-5.2*\h-\xh-0.8*\h)--(3.9*\l,-5.2*\h+\xh-0.8*\h);
                \draw[red] (4.8*\l,-5.2*\h+\xh-0.8*\h)--(6.0*\l,-5.2*\h-\xh-0.8*\h);
                \draw[red] (4.8*\l,-5.2*\h-\xh-0.8*\h)--(6.0*\l,-5.2*\h+\xh-0.8*\h);
                
                \draw[dotted] (0.6*\l,-5.2*\h-0.8*\h) -- ++(0,+1.4*\h) node[above]{$\tbg$};
                \draw[dotted] (6.0*\l,-5.2*\h) -- ++(0,0.6*\h) node[above]{$\tdo$};
                \draw[dotted] (4.9*\l,-5.2*\h+0.2*\h-0.8*\h) -- ++(0,-\xh-0.8*\h) node[below]{$\tdt$};
                \node at (7*\l,1.2*\h-6.2*\h) {$\cdo = 2$};
                \node at (7*\l,0.2*\h-6.2*\h) {$\cdt = 3$};

                \node[draw] at (-1.5*\l,0.6*\h) {scenario $\omega_1$};
                \node[draw] at (-1.5*\l,-2.7*\h) {scenario $\omega_2$};
                \node[draw] at (-1.5*\l,-5.7*\h) {scenario $\omega_3$};

            \node[draw,rounded corners = 0.10cm, align = left] at (11,-2.6*\h) {
                \tikz{ 
                    \draw[line width = 0.4mm] (0,0) -- (\l,0);
                } 
                non rescheduled job \\
                \tikz{
                    \draw[line width = 0.4mm] (0,0) -- (\l,0);
                    \draw[red] (0,\xh) -- (\l,-\xh);
                    \draw[red] (0,-\xh) -- (\l,\xh);
                }
                rescheduled job
            };
            \end{tikzpicture}
    \caption{Partial shift $\sh$ with five jobs with rescheduled and non-rescheduled first job under two different scenarios}
    \label{fig:ExModelResource}
\end{figure}


We now explain the rationale behind the constraints that elements of $\sbg\times\sdo\times\sdt$ must satisfy to be in $S$. 
Let $P$ be a path decorated with $\res$ for $\omega$, and let $\sh$ be the corresponding partial shift.
If there are two successive jobs $j$ and $j'$ in $\sh$ such that $\xe \leq \xb$, then consider $j$ and $j'$ to be the first pair of such jobs, then necessarily $\tdo = \tdt$, as illustrated by scenarios~$\omega_1$ and $\omega_2$ on Figure~\ref{fig:ExModelResource}. 
In that case, $\cdo = \cdt - 1$ if there is an odd number of jobs strictly before $j'$ in $\sh$, and $\cdo = \cdt$ is there is an even number.
These two possibilities are illustrated by scenarios $\omega_1$ and $\omega_2$ on Figure~\ref{fig:ExModelResource}.
If there are no two successive jobs $j$ and $j'$ in $\sh$ such that $\xe[\sh] \leq \xb[sh]$,
then either there is an even number of jobs in $\sh$, and we have $\tdo > \tdt$ and $\cdo = \cdt$, or there is an odd number, and we have $\tdo < \tdt$ and $\cdo = \cdt - 1$, as illustrated by scenario $\omega_3$ on Figure~\ref{fig:ExModelResource}.
This result is explicitly proven in the appendix, for a more detailed explanation see the proof of Proposition~\ref{prop:wellScheduled}.


\paragraph{}
We define the operator $\plus$ on $\monoidS$ as follows: \begin{align*}
     q \plus e = e\plus q &= q, \enskip \text{ for all } q \in \monoidS\\
     q \plus \infty = \infty \plus q &= \infty, \enskip \text{ for all } q \in \monoidS\\
     \res[a] \oplus \res[b]
      &= \res  
      \end{align*}
      where $\rbg = \rbg[a]$, and
      \begin{alignat*}{3}
      &
      \left\{ 
      \begin{array}{l}
      \coutUn = \coutUn_{a} + \coutUn_{b}, \\
      \heurefinUn = \heurefinUn_{b}, 
      \end{array}
      \right.
      \text{if }\heurefinUn_{a} \leq \heuredeb_{b}, 
      & \quad \text{and} \quad
      \left\{ 
      \begin{array}{l}
      \cdo = \cdo_{a} + \cdt_{b},\\
      \tdo = \tdt_{b},  
      \end{array}
      \right.
      & \text{ otherwise,}\\
      \text{and} \quad
      &\left\{ 
      \begin{array}{l}
      \cdt = \cdt_{a} + \cdo_{b}, \\
      \tdt = \tdo_{b},
      \end{array}
      \right. 
      \text{if } \tdt_{a} \leq \tbg_{b}
      & \quad \text{and} \quad
      \left\{ 
      \begin{array}{l}
      \cdt = \cdt_{a} + \cdt_{b},\\
      \tdt = \tdt_{b}, 
      \end{array}
      \right. 
      &  \text{otherwise.}
      \end{alignat*}
Remark that $\plus $ is not commutative.
The different cases in the definition of operator $\plus$ depend on the possibility to operate the first job of the second sequence after the last non rescheduled job of the first sequence.
\begin{lemma} \label{lemma:Monoid}
$(\monoidS,\plus)$ is a monoid.
\end{lemma}
The proof, which is available in appendix, shows that $\monoidS$ is stable by $\plus$ and the associativity of~$\plus$.
We define the operator $\moins$ on $\monoidS$ as follows.
\begin{align*}
    e \moins q \enskip \text{ and } \enskip q\moins \infty \enskip &\text{ for all } q \in \monoidS, \\
    \res[a] \moins \res[b]
    & \enskip \text{ if } \enskip \left \{ \begin{array}{rl}
        \heuredeb_{a} &\geqslant \heuredeb_{b}\\
        (\coutUn_{a},\heurefinUn_{a})&\moinslex (\coutUn_{b},\heurefinUn_{b})\\
        (\coutNun_{a},\heurefinNun_{a})&\moinslex (\coutNun_{b},\heurefinNun_{b})
    \end{array}
    \right.
\end{align*}
where $\moinslex$ corresponds to the lexicographical order on $\bbZ_+^{2}$.
The operator $\moins$ ensures that a resource with lower cost both when the first job is rescheduled and when it is not, is considered as lower. 
\begin{lemma} \label{lemma:OrderCompatible}
Order $\moins$ is compatible with the operator $\plus$.
\end{lemma}

\paragraph{}
We denote by $\meet_{lex}$ the lexicographical minimum on $\bbZ_+^{2}$.
Let $\meet$ be the binary operator on $\monoidS$ defined by,
\begin{equation*}
\begin{array}{c}
    e \meet q = q \meet e = e \quad \text{and} \quad 
    \infty \meet q = q \meet \infty = q, \enskip \text{ for all }q \in \monoidS\\
    \text{ and }
    \res[a] \meet \res[b] = \res \enskip \text{where} \enskip \left\{
    \begin{array}{l}
    \heuredeb = \max(\heuredeb_{a},\heuredeb_{b}), \\
    (\heurefinUn,\coutUn)= (\heurefinUn_{a},\coutUn_{a}) \meet_{lex} (\heurefinUn_{b},\coutUn_{b}),\\
    (\heurefinNun,\coutNun)= (\heurefinNun_{a},\coutNun_{a}) \meet_{lex} (\heurefinNun_{b},\coutNun_{b}).
    \end{array}
    \right.
\end{array}
\end{equation*}

\begin{lemma} \label{lemma:Meet}
$(\monoidS,\moins)$ is a lattice with meet operator $\meet$.
\end{lemma}
The proof shows that $\monoidS$ is stable by the operator $\meet$, and that the latest defines a greatest lower bound for any pair $(q,q')$ of elements of $\monoidS$. See appendix for more details.
The following theorem is an immediate corollary of the four previous lemmas.
\begin{theo}
$(\monoidS, \plus, \moins)$ is a lattice ordered monoid with meet operator $\meet$.
\end{theo}



\subsubsection{Full lattice ordered monoid}
\label{sub:full_lattice_ordered_monoid}

To model the pricing subproblem, we use a resource in $\monoidS$ for each scenario $\omega$ in $\Omega$, and a resource $\lambda \in \bbR$ to model the wage costs and the reduced costs. 
As $(\bbR,+, \leqslant)$ is a lattice ordered monoid with meet operator $\min$, the set $M=\monoidS^{ \Oscenario} \times \bbR$ endowed with the componentwise sum and order 
 is a lattice ordered monoid as a product of two lattice ordered monoids.

\subsubsection{Resources on the arcs}
Let $a \in A$ be an arc of $D$. We define for each scenario $\oscenario \in \Oscenario$ a resource $q_{a}^{\oscenario}$ on the arc a. The resource $r_a$ of an arc $a$ is of the form $((q_{a}^{\oscenario})_{\oscenario \in \Oscenario},\lambda_a)$, with 
$$
q_a^{\omega} = \left\{
\begin{array}{ll}
    \left(\begin{array}{c}
            \mathsf{bg}\big(\xb[j](\omega)\big)\\
            \mathsf{do}_a(\omega)\\
            \mathsf{dt}\big(1,-\infty\big)
    \end{array}\right)
  & \text{ if }  a = \big((j,\hb,\el),\cdot\big),\\
e & \text{ otherwise. }
\end{array}
\right.
$$
where
$$
\mathsf{do}_a(\omega) = \left\{
\begin{array}{ll}
\mathsf{do}\big(1,-\infty\big) & \text{if }\vl[j](\omega) = 1 \\
\mathsf{do}\big(0,\xe[j](\omega)\big) & \text{if }\vl[j](\omega) = 0 \text{ and } 
    \left\{
    \begin{array}{rl}
    a =& \big((j,\hb,\el),(j',\hb,\el)\big),\\
    a=& \big((j,\hb,\bl),\he\big) \text{ and } \he < \tbl + \tml \\
    \text{or } a=& \big((j,\hb,\al),\he\big) 
    \end{array}
    \right. \\
\mathsf{do}\big(0,\xe[j](\omega)+\tbr\big) & \text{if }\vl[j](\omega) = 0 \text{ and } 
    \left\{
    \begin{array}{rl}
    a =& \big((j,\hb,\bl),(j',\hb,\al)\big), \\
    \text{or }a=& \big((j,\hb,\bl),\he\big) \text{ and } \he > \tbl + \tml. 
    \end{array}
    \right.
\end{array}
\right.
$$
\noindent and
$$ \lambda_a = \left\{
\begin{array}{ll}
    - \lambda_j & \text{if $a$ is of the form $\big((j,\hb,\cdot),(j',\hb,\cdot)\big)$}, \\
    \cw(\he - \hb) - \lambda_j & \text{if $a$ is of the form $\big((j,\hb,\cdot),\he\big)$}, \\
    0 & \text{otherwise.}
\end{array}
\right. $$

\begin{proposition}\label{prop:FinalCost}
Let $P$ be an $o$-$d$ path and $\sh$ the corresponding shift. Then $$\bigoplus_{a \in P}r_a= \left(
\left(\begin{array}{c}
            \mathsf{bg}\big(\tbg_{P}(\omega)\big)\\
            \mathsf{do}\big(\cdo_{P}(\omega),\tdo_{P}(\omega)\big)\\
            \mathsf{dt}\big(\cdt_{P}(\omega),\tdt_{P}(\omega)\big)
    \end{array}\right)_{\oscenario \in \Oscenario}
, \cw(\he[\sh] - \hb[\sh]) - \sum_{j \in \sh} \lambda_j \right) ,
$$
where 
$\tbg_{P}(\omega)$ is the time at which the first job of $\sh$ starts under $\omega$,
$\tdo_{P}(\omega)$ the time at which the last job of $\sh$ under $\omega$ ends if non-rescheduled, else has value $-\infty$,
$\tdt_{P}(\omega)$ the time at which the last job of $\sh$ under $\omega$ would have end is the first job of $\sh$ was rescheduled under $\omega$ and if the last job is non-rescheduled, else has value $-\infty$,
$\cdo_{P}(\omega)$ is the number of rescheduled jobs in $\sh$ under $\omega$,
and $\cdt_{P}(\omega)$ the number of jobs of $\sh$ that would have been rescheduled if the first job of $\sh$ was rescheduled under $\omega$.
\end{proposition}
\begin{remark}
The condition $\mathsf{dt}\big(1,-\infty\big)$ and $\mathsf{do}\big(1,-\infty\big) \text{ if }\vl[j](\omega) = 1$ ensures that Hypothesis~\ref{hyp:propagation} is always satisfied, i.e.,~that a job operated after a rescheduled job is always realizable by the initial agent, since $-\infty \leqslant \mathsf{bg}\big(\xb[j'](\omega)\big)$, for every job $j'$ and scenario $\oscenario \in \Oscenario$.
\end{remark}

\subsubsection{Cost function}
Given $r =\big((q^{\oscenario})_{\oscenario \in\Oscenario},\lambda\big) \in \monoidS^{\Oscenario} \times \bbR$, we define
\begin{align*}
    c(r) = \frac{\cbu}{\sharp \Oscenario}\sum_{\oscenario \in \Oscenario}c_{\monoidS}(q^{\oscenario}) + \lambda, 
\end{align*}
where $c_{\monoidS}(e)= 0$, $c_{\monoidS}(\infty)= \infty$, and
$c_{\monoidS}\left(\res[\omega]\right)= \cdo_{\oscenario}$.\\
Finally, we define $\rho(r) = 0$ for all $r\in \monoidS^{ \Oscenario} \times \bbR$.
The following theorem concludes the reduction of the pricing subproblem to the  \MRCSP.
\begin{theo} \label{prop:FinalProp}
There is a bijection between $o$-$d$ paths in $D$ and well-scheduled shifts. Furthermore, given an $o$-$d$ path $P$ and the corresponding shift $\sh$, we have
$$c\left(\bigoplus_{a \in P}r_a\right) = c_{\sh} -\sum_{j \in \sh} \lambda_j. $$
The sequence of job $\task_{1}, \task_{2}, \ldots, \task_{k}$ is a feasible shift if and only if the corresponding path $p$ is an o-d path in $D$ whose resource $q_{p}$ satisfies $\rho(q_{p})=0$. Furthermore, in that case, $c(q_{p}) = \costshift - \sum_{\task \in \shift} \lambda_{\task}$.
\end{theo}

\bigskip\bigskip\bigskip\bigskip

\subsubsection*{Acknowledgments}
\label{ssub:acknowledgments}
\paragraph{
We are grateful to Air France which partially
supported the project. We are especially thankful to Isabel Gomez, Beno\^it Robillard, and Blaise Brigaud, who initiated the
project, and to Pierre de Fr\'eminville who helped us with the implementation in Air France softwares.}


\include{conclu}

\bibliographystyle{plainnat}
\bibliography{ground}
 

\begin{appendix}{\textbf{Proofs}}

\begin{proof}[Proof of Proposition~\ref{prop:wellScheduled}.]
Let $\sh$ be a feasible shift containing breaks, and $\sh'$ be the shift obtained by removing all its breaks but one, denoted $j$, and replacing $j$ by the break $[t - \tbr, t]$, where
$$t = 
\left\{
\begin{array}{ll}
\min(\te[\sh],\tel) & \text{if $j$ is the last activity of $\sh$,} \\
\min(\tb[j'],\tel) & \text{otherwise, where $j'$ is the activity after $j$ in $\sh$.}
\end{array}
\right.
 $$
 Then $\sh'$ is a well-scheduled shift containing the same jobs as $\sh$. 
 Furthermore, $\sh$ and $\sh'$ have the same beginning and ending time.
 Removing a break from a shift can only decrease the number of rescheduled job. 
Indeed, suppose that we remove a break $\br$ between jobs $j_i$ and $j_{i+1}$.
If $j_i$ is rescheduled or if we do not have $\xe[j_i] \leq \xb[j_{i+1}] < \xe[j_i] + \tbr$, then removing $\br$ does not change the rescheduled jobs. 
Suppose that $j_i$ is not rescheduled and $\xe[j_i] \leq \xb[j_{i+1}] < \xe[j_i] + \tbr$, and let $j_{i+1}, \ldots, j_{i+m'}$ be the jobs after $j_i$ in $\sh$. 
Let $m$ be equal to the smallest integer in $[k-1]$ such that either there is no break between $j_{i+m}$ and $j_{i+m+1}$ and $\xe[j_{i+m}] \leq \xb[j_{i+m+1}]$, or there is a break between $j_{i+m}$ and $j_{i+m+1}$ and $\xe[j_{i+m}] \leq \xb[j_{i+m+1}] - \tbr$, and to $m'$ otherwise.
Then only the rescheduling of jobs $j_{i+k}$ with $k$ in $[m]$ is affected when $\br$ is removed.
Indeed, jobs $j_{i+k}$ with $k$ odd in $[m]$ are rescheduled if and only if $\br$ is not removed, and jobs $j_{i+k}$ with $k$ even in $[m]$ are rescheduled if and only $\br$ is removed.
Hence, removing $\br$ decreases the number of rescheduled job by one if $m$ is odd, and does not change the number of rescheduled jobs if $m$ is even.
Hence, $\sh'$ has the same beginning and ending times as $\sh$, and at most as many rescheduled jobs, which gives $c_{\sh'} \leq c_{\sh}$.
\end{proof}

\begin{proof}[Proof of Proposition~\ref{prop:bijectionDigraph}.]
First we are going to prove that every $o$-$d$ path in $D$ represents a shift as defined in Definition~\ref{def:shifts}.
Let $P$ be an $o$-$d$ path in $D$. Path $P$ is defined by an ordered sequence of vertices of the form $\{o\}$, $\big\{(j^i,\hb,e)$ for $i$ in $\{1,\ldots,k\},e \in\{\bl,\al\} \text{ and some } \hb \text{ in } \calHb \big\}$, $\{\he\}$ and $\{d\}$, where $\he \in \calHf$. 
As any path intersects a unique $V_{\hb}$, the component $\hb$ in $\calHb$ is identical for all the vertices in $P$. 
$P$ thus represents the sequence $\tb[\sh],\task^{1},\task^{2}, \ldots, \task^{k}, \te[\sh]$, with $\tb[\sh] = \hb$ and $\te[\sh]=\he$. In addition we add to that sequence a break denoted $\task_{b}$ in the two following cases:
\begin{itemize}
    \item if there exists an edge of the form $\big((\task^{i},\hb,\bl), (\task^{i+1},\hb,\al)\big)$, we add it between $\task^{i}$ and $\task^{i+1}$, and define $\te[j_{b}] = \min(\tb[j^{i+1}], \tel)$ and $\tb[j_{b}] = \te[j_{b}] - \tbr$.
    \item if there exists an edge of the form $\big((\task^{k},\hb,\bl),\he\big)$ with $\he \geq \tbl + \tml$, we add it after $\task^{k}$, and define $\te[j_{b}] = \min(\te[\sh], \tel)$ and $\tb[j_{b}] = \te[j_{b}] - \tbr$.
\end{itemize} 
The pair composed of two successive vertices $(\task^{i},\hb,e)$ and $(\task^{i+1},\hb,e)$ in $P$ is an arc of $D$. The definition of arcs in $D$ ensures both $\hbshift \leq\tb[\task^{i}] \text{ for all $i$ in } \{1,\ldots,k\}$ and $\te[j^i] \leq \tb[j^{i+1}] \text{ for all $i$ in } \{1,\ldots,k-1\}$. 
Furthermore the existence of arc $\big((j^k,\hb,\el), \he \big)$, where $\el \in\{\bl,\al\}$, proves that $\te[j^{k}] \leq \hfshift$.
Let us now consider the case where a lunch break $\task_{b}$ has been added to the sequence. If $\task_{b}$ has been added between two other jobs $\task^{i}$ and $\task^{i+1}$, it means that the edge between vertices $(\task^{i},\hb,\bl)$ and $(\task^{i+1},\hb,\al)$ exists and thus we have $\te[j^{i}] \leqslant \tb[j_{b}]$ and $\te[j_{b}] \leqslant \te[j_{i+1}]$, considering how we define $\tb[j_{b}]$ and $\te[j_{b}]$. In addition all the other inequalities still hold. If $\task_{b}$ has been added after the last job $\task^{k}$ of the sequence, which ensures that the edge $\big((\task^{k},\hb,\bl),\te[\sh]\big)$ exits. According to how we choose the beginning and ending time of such a break, since the previous edge exists, we have $\te[j^{k}] \leqslant \tb[j_{b}]$ and $\te[j_{b}] \leqslant \te[\sh]$. In addition all the other inequalities still hold.
The sequence $\tb[\sh],\task^{1},\task^{2}, \ldots, \task^{k}, \te[\sh]$, with eventually a break added, represents thus a shift $\sh$ according to Definition~\ref{def:shifts}. \\
Second we are going to show that every $o$-$d$ path in $D$ represents a feasible shift according to rules~\ref{rule:Amplitude} and~\ref{rule:LunchBreak}. 
Rule~\ref{rule:Amplitude} is guaranteed by the fact that edges of type $\big((j,\hb,\el),\he\big)$, where $\el \in\{\bl,\al\}$, exist if and only if $\he \leq \hb + \tm$. 
Rule~\ref{rule:LunchBreak} is guaranteed by the use of $\bl$ and $\al$ which corresponds for each beginning time $\hb$ to respectively the jobs operated before the lunch break and after. Thus if $|[b_{lunch}, f_{lunch}] \cap [\hbshift,\hfshift]| \geq \tml$, that means that either an edge of the form $\big((\task^{i},\hb,\bl), (\task^{i+1},\hb,\al)\big)$ exists, or the last edge of path $P$ is of the form $\big((\task^{k},\hb,\bl),\te[\sh]\big)$, and we have $\te[\sh] \geq \tbl + \tml$. As both cases lead to the insertion of a lunch break in the sequence, and as we have already mentioned, an edge between a $\bl$ element and an $\al$ element exists only if there is at least time $\tbr$ between the two sequencing jobs or between the last job and the end of the shift, \ref{rule:LunchBreak} is respected.\\
Finally we are going to prove that every $o$-$d$ path in $D$ corresponds to a unique feasible well-scheduled shift as defined in Part~\ref{sub:well_scheduled_breaks}. If the feasible shift $\sh$ doesn't contain a break, then it is a well scheduled shift. Let us focus on the case where at least a break $j_{b}$ has been introduced in the sequence of jobs. Regarding how we defined its beginning and ending times, they respect the definition of well-scheduled break. In addition the shift built contains at most one lunch break as there is at most one edge between a $\bl$ element and an $\al$ element, or with the end of the shift, and is thus a well-scheduled shift. Finally, since a shift is uniquely defined by its sequence of jobs, including the lunch break, and its beginning and ending times, the uniqueness is immediately proved.\\
Let us show that every feasible well-scheduled shift $\sh$ represented by the following sequence $\tb[\sh],\task^{1},\task^{2}, \ldots, \task^{k}, \te[\sh]$ is associated to a unique $o$-$d$ path $P$ in $D$, defined by the following rule. First set $\hb = \tb[\sh]$ and $\he = \te[\sh]$, and add to path $P$ the vertex $(\he)$. Then browse the sequence in increasing order and for every job $\task^{i}$ different from a break: 
\begin{itemize}
    \item if $\task^{i} \in \Jbl \cap \Jal$, if we had already encounter a break, then add to the path $P$ the vertex $(\task^{i},\hb,\al)$, if not, the vertex $(\task^{i},\hb,\bl)$.
    \item else if $\task^{i} \in \Jbl$, add to the path $P$ the vertex $(\task^{i},\hb,\bl)$
    \item else if $\task^{i} \in \Jal$, add to the path $P$ the vertex $(\task^{i},\hb,\al)$
\end{itemize} 
Using the same reasoning as the one we used to show that an $o$-$d$ path represents a well-scheduled feasible shift, we can show that every pair of two successive vertices of $P$ is an arc, and thus prove that $P$ is an $o$-$d$ path in $D$. As a path in $D$ is uniquely represented by its succession of vertices, then the uniqueness of $P$ is also proved, which concludes the proof.
\end{proof}

\medbreak

Proposition~\ref{prop:MRCSPalgoConvergence} is proved in  \citet{2017arXiv170606901P}.

\medbreak

\begin{proof}[Proof of lemma~\ref{lemma:Monoid}.] \label{proof:IsaMonoid}
For the purpose of this proof we introduce the following equations.
\[
    \tdo < \tdt \enskip \Rightarrow \cdo = \cdt \label{MonoidP1} \tag{P1}
\]
\[
\tdt < \tdo \enskip \Rightarrow \cdo = \cdt - 1 \label{MonoidP2} \tag{P2}
\]
\[
    \tdo = \tdt \enskip \Rightarrow \cdo = \cdt \text{ or } \cdo = \cdt -1 \label{MonoidP3} \tag{P3}
\]
Let us prove that $\monoidS$ is stable by $\plus$.
Let $q_{a},q_{b} \in \monoidS$, and $x = q_{a} \plus q_{b}$.
We prove the result by disjunction of cases.
If $q_{a} \in \{e,\infty\}$ or $q_{b} \in  \{e,\infty\}$, then by definition of $\plus$, the result is immediate.\\
Suppose now that $q_{a} = \res[a]$, $q_{b} = \res[b]$ and $x = \res$. 
\begin{itemize}
    \item Let us prove that $x$ satisfies \eqref{MonoidP1}. Let us assume that we have $\tdo < \tdt\enskip (H)$. We deduce that $\tdo_{b} \not = \tdt_{b}$
\begin{itemize}
	\item[+] if $\tdo_{b} < \tdt_{b} \enskip (h)$, we have 
	\begin{align*}
	\tdo &= \tdo_{b} \enskip \text{ and thus } \tdo_{b} \leq \tbg_{b} \\
	\tdt &= \tdt_{b} \enskip \text{ and thus } \tdt_{a} > \tbg_{b}\\
	\cdo &= \cdo_{a} + \cdo_{b}\\
	\cdt &= \cdt_{a} +\cdt_{b}\\
	\end{align*}
	With the two first inequalities, we deduce that $\tdo_{a} < \tdt_{a}$. With that inequality and $(h)$, as $q_{a}, q_{b} \in \monoidS$, with property \eqref{MonoidP1} we have : $\cdo_{a} = \cdt_{a}$ and  $\cdo_{b} = \cdt_{b}$.\\
	Thus : $\boxed{\cdo = \cdt}$.
	\item[+] else $\tdo_{b} > \tdt_{b} \enskip (h)$, and we have
	\begin{align*}
	\tdo &= \tdt_{b} \enskip \text{ and thus } \tdo_{a} > \tbg_{b} \\
	\tdt &= \tdo_{b} \enskip \text{ and thus } \tdt_{a} \leq \tbg_{b}\\
	\cdo &= \cdo_{a} + \cdt_{b}\\
	\cdt &= \cdt_{a} + \cdo_{b}\\
	\end{align*}
	With the two first inequalities, we deduce that $\tdt_{a} < \tdo_{a}$. With that inequality and $(h)$, as $q_{a}, q_{b} \in \monoidS$, with property \eqref{MonoidP2} we have : $\cdo_{a} = \cdt_{a} -1 $ and  $\cdo_{b} = \cdt_{b} -1$.\\
	Thus : $\boxed{\cdo = \cdt}$.
\end{itemize}
    \item Let us prove that $x$ satisfies \eqref{MonoidP2}. Let us assume that we have $\tdt < \tdo$. Thus we can already deduce that $\tdo_{b} \not = \tdt_{b}$
\begin{itemize}
	\item[+] if $\tdo_{b} < \tdt_{b} \enskip (h)$, we have 
	\begin{align*}
	\tdo &= \tdt_{b} \enskip \text{ and thus } \tdo_{a} > \tbg_{b}\\
	\tdt &= \tdo_{b} \enskip \text{ and thus } \tdt_{a} \leq \tbg_{b} \\
	\cdo &= \cdo_{a} + \cdo_{b}\\
	\cdt &= \cdt_{a} + \cdo_{b}\\
	\end{align*}
	With the two first inequalities, we deduce that $\tdt_{a} < \tdo_{a}$. As $q_{a} \in \monoidS$, with property \eqref{MonoidP2} we have : $\cdo_{a} = \cdt_{a} -1$. With $(h)$, as $q_{b} \in \monoidS$, with \eqref{MonoidP1} we have: $\cdo_{b} = \cdt_{b}$.\\
	Thus : $\boxed{\cdo = \cdt-1}$.
	\item[+] else $\tdo_{b} > \tdt_{b} \enskip (h)$, and we have
	\begin{align*}
    \tdo &= \tdo_{b} \enskip \text{ and thus } \tdo_{a} \leq \tbg_{b} \\
	\tdt &= \tdt_{b} \enskip \text{ and thus } \tdt_{a} > \tbg_{b}\\
	\cdo &= \cdo_{a} + \cdo_{b}\\
	\cdt &= \cdt_{a} +\cdt_{b}\\
	\end{align*}
	With the two first inequalities, we deduce that$\tdo_{a} < \tdt_{a}$. As $q_{a} \in \monoidS$, with property \eqref{MonoidP1} we have:$\cdo_{a} = \cdt_{a}$.With $(h)$, as $q_{b} \in \monoidS$, with \eqref{MonoidP2} we have: $\cdo_{b} = \cdt_{b} -1$.\\
	Thus : $\boxed{\cdo = \cdt-1}$.
\end{itemize}
    \item Let us prove that $x$ satisfies \eqref{MonoidP3}. Assume that we have $\tdt = \tdo$.
    \begin{itemize}
        \item[+] if  $\tdt = \tdo = \tdo_{b}$, we have
        \begin{align*}
        \tdo_{a} &\leq \tbg_{b} \enskip \text{ and thus } \cdo = \cdo_{a} + \cdo_{b}\\
    	\tdt_{a}  & \leq \tbg_{b} \enskip \text{ and thus } \cdt = \cdt_{a} +\cdo_{b}\\
    	\end{align*}
    	As $q_{a} \in \monoidS$, $q_{a}$ satisfies \eqref{MonoidP3}, and thus $\boxed{x \text{ satisfies \eqref{MonoidP3}}}$.
    	\item[+] else $\tdt = \tdo = \tdt_{b}$, and we have
    	\begin{align*}
        \tdo_{a} &> \tbg_{b} \enskip \text{ and thus } \cdo = \cdo_{a} + \cdt_{b}\\
    	\tdt_{a}  & > \tbg_{b} \enskip \text{ and thus } \cdt = \cdt_{a} +\cdt_{b}\\
    	\end{align*}
    	As $q_{a} \in \monoidS$, $q_{a}$ satisfies \eqref{MonoidP3}, and thus $\boxed{x \text{ satisfies \eqref{MonoidP3}}}$.
    \end{itemize}
\end{itemize}
As $x$ satisfies \eqref{MonoidP1}, \eqref{MonoidP2} and \eqref{MonoidP3}, $x \in \monoidS$ and this concludes the stability of $\plus$.
\\
We now prove the associativity of $\plus$.\\
Let $q_{a},q_{b},q_{c} \in \monoidS$, and $q_{t} = q_{a} \plus (q_{b} \plus q_{c})$, $q_{u} = (q_{a} \plus q_{b}) \plus q_{c}$.\\
If $q_{a} \in \{e,\infty\}$ or $q_{b} \in  \{e,\infty\}$ or $q_{v} \in  \{e,\infty\}$, then by definition of $\plus$, result is immediate.\\
Else, let $q_{a} = \res[a]$, $q_{b} = \res[b]$, $q_{c} = \res[c]$, $q_{t} = \res[t]$, $q_{u} = \res[u]$.\\
By defining $q_{f} = q_{b} \plus q_{c} = \res[f]$, we have $q_{t} = q_{a} \plus \res[f]$, with $\tbg_{f} = \tbg_{b}$ and
\[\tdo_{f} = 
\left\{ \begin{array}{l l}
  \tdo_{c} & \quad \text{if } \tdo_{b} \leq \tbg_{c}\\
  \tdt_{c} & \quad \text{otherwise}\\ \end{array} \right.
 \text{ , } \tdt_{f} = 
 \left\{ \begin{array}{l l}
 \tdo_{c} & \quad \text{if } \tdt_{b} \leq\tbg_{c}\\
  \tdt_{c} & \quad \text{otherwise}\\ \end{array} \right.\]

  \[\cdo_{f} = 
\left\{ \begin{array}{l l}
  \cdo_{b} +\cdo_{c} & \quad \text{if } \tdo_{b} \leq \tbg_{c}\\
  \cdo_{b} + \cdt_{c}  & \quad \text{otherwise}\\ \end{array} \right.
 \text{ , } \cdt_{f} = 
 \left\{ \begin{array}{l l}
  \cdt_{b} +\cdo_{c} & \quad \text{if } \tdt_{b} \leq\tbg_{c}\\
  \cdt_{b} + \cdt_{c} & \quad \text{otherwise}\\ \end{array} \right.\]

Hence $\tbg_{t} = \tbg_{a}$ and
\[\tdo_{t} = 
\left\{ \begin{array}{l l}
  \tdo_{f} & \quad \text{if } \tdo_{a} \leq \tbg_{b}\\
  \tdt_{f} & \quad \text{otherwise}\\ \end{array} \right.
 \text{ , } \tdt_{t} = 
 \left\{ \begin{array}{l l}
  \tdo_{f} & \quad \text{if } \tdt_{a} \leq \tbg_{b}\\
  \tdt_{f} & \quad \text{otherwise}\\ \end{array} \right.\]

  \[\cdo_{t} = 
\left\{ \begin{array}{l l}
  \cdo_{a} + \cdo_{f} & \quad \text{if }  \tdo_{a} \leq \tbg_{b}\\
  \cdo_{a} + \cdt_{f}  & \quad \text{otherwise}\\ \end{array} \right.
 \text{ , } \cdt_{t} = 
 \left\{ \begin{array}{l l}
  \cdt_{a} + \cdo_{f} & \quad \text{if }\tdt_{a} \leq \tbg_{b}\\
  \cdt_{a} + \cdt_{f} & \quad \text{otherwise}\\ \end{array} \right.\]
 \\
 By defining $q_{a} \plus q_{b} = \res[g]$, we have $q_{u} = \res[g] \plus q_{c}$, with $\tbg_{g} = \tbg_{a}$ and
\[\tdo_{g} = 
\left\{ \begin{array}{l l}
  \tdo_{b} & \quad \text{if } \tdo_{a} \leq \tbg_{b}\\
  \tdt_{b} & \quad \text{otherwise}\\ \end{array} \right.
 \text{ , } \tdt_{g} = 
 \left\{ \begin{array}{l l}
  \tdo_{b} & \quad \text{if } \tdt_{a} \leq \tbg_{b}\\
  \tdt_{b} & \quad \text{otherwise}\\ \end{array} \right.\]

  \[\cdo_{g} = 
\left\{ \begin{array}{l l}
  \cdo_{a} + \cdo_{b} & \quad \text{if } \tdo_{a} \leq \tbg_{b}\\
  \cdo_{a} + \cdt_{b}  & \quad \text{otherwise}\\ \end{array} \right.
 \text{ , } \cdt_{g} = 
 \left\{ \begin{array}{l l}
  \cdt_{a} + \cdo_{b} & \quad \text{if } \tdt_{a} \leq \tbg_{b}\\
  \cdt_{a} + \cdt_{b} & \quad \text{otherwise}\\ \end{array} \right.\]

Hence $\tbg_{u} = \tbg_{a}$, and
\[\tdo_{u} = 
\left\{ \begin{array}{l l}
  \tdo_{c} & \quad \text{if } \tdo_{g} \leq \tbg_{c}\\
  \tdt_{c} & \quad \text{otherwise}\\ \end{array} \right.
 \text{ , } \tdt_{u} = 
 \left\{ \begin{array}{l l}
  \tdo_{c} & \quad \text{if } \tdt_{g} \tbg_{c}\\
  \tdt_{c} & \quad \text{otherwise}\\ \end{array} \right.\]

  \[\cdo_{u} = 
\left\{ \begin{array}{l l}
  \cdo_{g} + \cdo_{c} & \quad \text{if } \tdo_{g} \leq \tbg_{c}\\
  \cdo_{g} + \cdt_{c}  & \quad \text{otherwise}\\ \end{array} \right.
 \text{ , } \cdt_{u} = 
 \left\{ \begin{array}{l l}
  \cdt_{g} +\cdo_{c} & \quad \text{if } \tdt_{g} \leq \tbg_{c}\\
  \cdt_{g} + \cdt_{c} & \quad \text{otherwise}\\ \end{array} \right.\]

\bigbreak

\begin{itemize}
    \item From these equations, we deduce $\rbg[t]=\rbg[u]$.
    \item We now show that: $\rdo[t] = \rdo[u]$:
    \begin{itemize}
	\item[+] if $\tdo_{t} = \tdo_{f}$, then $\tdo_{a}\leq \tbg_{b}$,\\
	 and we have : $\cdo_{t}=\cdo_{a}+\cdo_{f}$, $\tdo_{g}=\tdo_{b}$, and $\cdo_{g}=\cdo_{a}+\cdo_{b}$.
	\begin{itemize}
		\item[+] if $\tdo_{f}=\tdo_{c}$, then $\tdo_{b}\leq \tbg_{c}$,\\
		 and so $\cdo_{f}=\cdo_{b}+\cdo_{c}$. Moreover $\tdo_{g}\leq \tbg_{c}$, thus $\boxed{\tdo_{u}=\tdo_{c}=\tdo_{t}}$, and $\cdo_{u}=\cdo_{g}+\cdo_{c}$, thus $\boxed{\cdo_{u}=\cdo_{a}+\cdo_{b}+\cdo_{c}=\cdo_{t}}$
		\item[+] else $\tdo_{f}=\tdt_{c}$, then $\tdo_{b}> \tbg_{c}$,\\
		 and so $\cdo_{f}=\cdo_{b}+\cdt_{c}$. Moreover $\tdo_{g}> \tbg_{c}$, thus $\boxed{\tdo_{u}=\tdt_{c}=\tdo_{t}}$, and $\cdo_{u}=\cdo_{g}+\cdt_{c}$, thus $\boxed{\cdo_{u}=\cdo_{a}+\cdo_{b}+\cdt_{c}=\cdo_{t}}$
	\end{itemize}
	\item[+] else $\tdo_{t} = \tdt_{f}$, then $\tdo_{a}> \tbg_{b}$,\\
	 and we have : $\cdo_{t}=\cdo_{a}+\cdt_{f}$, $\tdo_{g}=\tdt_{b}$, and $\cdo_{g}=\cdo_{a}+\cdt_{b}$.
	 \begin{itemize}
		\item[+] if $\tdt_{f}=\tdo_{c}$, then $\tdt_{b}\leq \tbg_{c}$,\\
		 and so $\cdo_{f}=\cdt_{b}+\cdo_{c}$. Moreover $\tdo_{g}\leq \tbg_{c}$, thus $\boxed{\tdo_{u}=\tdo_{c}=\tdo_{t}}$, and $\cdo_{u}=\cdo_{g}+\cdo_{c}$, thus $\boxed{\cdo_{u}=\cdo_{a}+\cdt_{b}+\cdo_{c}=\cdo_{t}}$
		\item[+] else $\tdt_{f}=\tdt_{c}$, then $\tdt_{b}> \tbg_{c}$,\\
		 and so $\cdt_{f}=\cdt_{b}+\cdt_{c}$. Moreover $\tdt_{g}> \tbg_{c}$, thus $\boxed{\tdo_{u}=\tdt_{c}=\tdo_{t}}$, and $\cdo_{u}=\cdo_{g}+\cdt_{c}$, thus $\boxed{\cdo_{u}=\cdo_{a}+\cdt_{b}+\cdt_{c}=\cdo_{t}}$
	\end{itemize}
    \end{itemize}
    \item We finally show that: $\rdt[t] = \rdt[u]$:
    \begin{itemize}
	\item[+] if $\tdt_{t} = \tdo_{f}$, then $\tdt_{a}\leq\tbg_{b}$,\\
	 and we have : $\cdt_{t}=\cdt_{a}+\cdo_{f}$, $\tdt_{g}=\tdo_{b}$, and $\cdt_{g}=\cdt_{a}+\cdo_{b}$.
	\begin{itemize}
		\item[+] if $\tdo_{f}=\tdo_{c}$, then $\tdo_{b}\leq\tbg_{b}$,\\
		 and so $\cdo_{f}=\cdo_{b}+\cdo_{c}$. Moreover $\tdo_{g}\leq\tbg_{b}$, thus $\boxed{\tdt_{u}=\tdo_{c}=\tdt_{t}}$, and $\cdt_{u}=\cdo_{g}+\cdo_{c}$, thus $\boxed{\cdt_{u}=\cdt_{a}+\cdo_{b}+\cdo_{c}=\cdt_{t}}$
		\item[+] else $\tdo_{f}=\tdt_{c}$, then $\tdo_{b}>\tbg_{b}$,\\
		 and so $\cdo_{f}=\cdo_{b}+\cdt_{c}$. Moreover $\tdo_{g}>\tbg_{b}$, thus $\boxed{\tdt_{u}=\tdt_{c}=\tdt_{t}}$, and $\cdo_{u}=\cdo_{g}+\cdt_{c}$, thus $\boxed{\cdt_{u}=\cdt_{a}+\cdo_{b}+\cdt_{c}=\cdt_{t}}$
	\end{itemize}
	\item[+] else $\tdt_{t} = \tdt_{f}$, then $\tdo_{a}>\tbg_{b}$,\\
	 and we have : $\cdt_{t}=\cdt_{a}+\cdt_{f}$, $\tdt_{g}=\tdt_{b}$, and $\cdt_{g}=\cdt_{a}+\cdt_{b}$.
	 \begin{itemize}
		\item[+] if $\tdt_{f}=\tdo_{c}$, then $\tdt_{b}\leq\tbg_{b}$,\\
		 and so $\cdt_{f}=\cdt_{b}+\cdo_{c}$. Moreover $\tdt_{g}\leq\tbg_{b}$, thus $\boxed{\tdt_{u}=\tdo_{c}=\tdt_{t}}$, and $\cdt_{u}=\cdt_{g}+\cdo_{c}$, thus $\boxed{\cdt_{u}=\cdt_{a}+\cdt_{b}+\cdo_{c}=\cdt_{t}}$
		\item[+] else $\tdt_{f}=\tdt_{c}$, then $\tdo_{b}>\tbg_{b}$,\\
		 and so $\cdt_{f}=\cdt_{b}+\cdt_{c}$. Moreover $\tdt_{g}>\tbg_{b}$, thus $\boxed{\tdt_{u}=\tdt_{c}=\tdt_{t}}$, and $\cdt_{u}=\cdt_{g}+\cdt_{c}$, thus $\boxed{\cdt_{u}=\cdt_{a}+\cdt_{b}+\cdt_{c}=\cdo_{t}}$
	\end{itemize}
\end{itemize}
\end{itemize}
To conclude we have proved that $q_{a}\plus (q_{b}\plus q_{c}) = (q_{a} \plus q_{b}) \plus q_{c}$ and thus the operator $\plus$ is associative.
\\
This concludes the proof: $(\monoidS,\plus)$ is a monoid.
\end{proof}

\medbreak

\begin{proof}[Proof of lemma~\ref{lemma:OrderCompatible}.]
For the purpose of this proof we use notations established in proof of Lemma~\ref{lemma:Monoid} and introduce the following equations.
\[
    \tbg_{t} \geqslant \tbg_{u} \label{Order1} \tag{O1}
\]
\[
    (\cdo_{t},\tdo_{t})\moinslex (\cdo_{u},\tdo_{u}) \label{Order2} \tag{O2}
\]
\[
    (\cdt_{t},\tdt_{t})\moinslex (\cdt_{u},\tdt_{u}) \label{Order3} \tag{O3}
\]

Let us prove that the order $\moins$ is compatible with the operator $\plus$.
Let $q_{a},q_{b} \in \monoidS$ be such as $q_{a} \moins q_{b}$.
Let $q_{x}$ be in $\monoidS$.\\
We are going to show by disjunction of cases that $q_{a} \plus q_{x} \moins q_{b} \plus q_{x}$. Let $q_{t}=q_{a} \plus q_{x}$, and $q_{u}=q_{b} \plus q_{x}$. We want to prove $q_{t} \moins q_{u}$.
If $q_{a} \in \{e,\infty\}$ or $q_{b} \in  \{e,\infty\}$, then by definition of $\moins$ and $\plus$, the result is immediate.\\
Else, let $q_{a} = \res[a]$, $q_{b} = \res[b]$, $q_{x} = \res[x]$, $q_{t} = \res[t]$, $q_{u} = \res[u]$.\\
\begin{itemize}
    \item Let us prove that $q_u$ and $q_t$ satisfy \eqref{Order1}. We have $\tbg_{t}=\tbg_{a}$ and $\tbg_{u}=\tbg_{b}$. As $q_{a} \moins q_{b}$, we directly have $\boxed{\tbg_{t} \geqslant \tbg_{u}}$.
    \item Let us prove \eqref{Order2} i.e.  $(\cdo_{t},\tdo_{t})\moinslex (\cdo_{u},\tdo_{u})$.
    \begin{itemize}
	\item[+] if $\tdo_{a} \leq\tbg_{x}$, then $\cdo_{t} = \cdo_{a}+\cdo_{x}$ and $\tdo_{t}=\tdo_{x}$
	\begin{itemize}
		\item[+] if $\tdo_{b} \leqslant \tbg_{x}$, then $\cdo_{u}=\cdo_{b}+\cdo_{x}$ and $\tdo_{u}=\tdo_{x}$. As $q_{a} \moins q_{b}$, with \eqref{Order2} $\cdo_{a} \leqslant \cdo_{b}$, and thus we have $\cdo_{t}\leqslant \cdo_{u}$. As $\tdo_{t}=\tdo_{u}$, it comes $\boxed{(\cdo_{t},\tdo_{t})\moinslex (\cdo_{u},\tdo_{u})}$
		\item[+] else $\tdo_{b} >\tbg_{x}$, then $\cdo_{u}=\cdo_{b}+\cdt_{x}$ and $\tdo_{u}=\tdt_{x}$.
		\begin{itemize}
			\item[+] if $\cdo_{x}< \cdt_{x}$, then we have $\cdo_{t}<\cdo_{u}$ and thus $\boxed{(\cdo_{t},\tdo_{t})\moinslex (\cdo_{u},\tdo_{u})}$.
			\item[+] else $\cdo_{x}=\cdt_{x}$, and as $q_{x} \in \monoidS$, \eqref{MonoidP2} ensures we have $\tdo_{x} \leqslant \tdt_{x}$, thus $\tdo_{t}\leqslant \tdo_{u}$. As $\cdo_{t} \leqslant \cdo_{u}$, it comes $\boxed{(\cdo_{t},\tdo_{t})\moinslex (\cdo_{u},\tdo_{u})}$.
		\end{itemize}
	\end{itemize}
	\item[+] else $\tdo_{a} >\tbg_{x}$, then $\cdo_{t} = \cdo_{a}+\cdt_{x}$ and $\tdo_{t}=\tdt_{x}$.
	\begin{itemize}
		\item[+] if $\tdo_{b} \leq\tbg_{x}$, then $\cdo_{u}=\cdo_{b}+\cdo_{x}$ and $\tdo_{u}=\tdo_{x}$. In this case we have $\tdo_{b}<\tdo_{a}$, and thus by definition of $\moins$, since $q_{a}\moins q_{b}$, $\cdo_{a}<\cdo_{b}$.
		\begin{itemize}
			\item[+] if $\cdo_{x}=\cdt_{x}$, then $\cdo_{t}<\cdo_{u}$ and thus $\boxed{(\cdo_{t},\tdo_{t})\moinslex (\cdo_{u},\tdo_{u})}$.
			\item[+] else  $\cdo_{x}=\cdt_{x}-1$, and as $c_{x} \in \monoidS$, \eqref{MonoidP1} ensures that we have $\tdt_{x}\leqslant \tdo_{x}$, thus $\tdo_{t}\leqslant \tdo_{u}$. We have $\boxed{(\cdo_{t},\tdo_{t})\moinslex (\cdo_{u},\tdo_{u})}$.
		\end{itemize}
		\item[+] else $\tdo_{b} >\tbg_{x}$, then $\cdo_{u}=\cdo_{b}+\cdt_{x}$ and $\tdo_{u}=\tdt_{x}$. As $q_{a}\moins q_{b}$, we are sure that $\cdo_{t}\leqslant \cdo_{u}$ with \eqref{Order2}, and as $\tdo_{u}=\tdo_{t}$ we effectively have $\boxed{(\cdo_{t},\tdo_{t})\moinslex (\cdo_{u},\tdo_{u})}$.
	\end{itemize}
    \end{itemize}
    \item Let us prove \eqref{Order3} i.e $(\cdt_{t},\tdt_{t})\moinslex (\cdt_{u},\tdt_{u})$.
    \begin{itemize}
	\item[+] if $\tdt_{a} \leqslant \tbg_{x}$, then $\cdt_{t} = \cdt_{a}+\cdo_{x}$ and $\tdt_{t}=\tdo_{x}$
	\begin{itemize}
		\item[+] if $\tdt_{b} \leqslant \tbg_{x}$, then $\cdt_{u}=\cdt_{b}+\cdo_{x}$ and $\tdt_{u}=\tdo_{x}$. As $q_{a} \moins q_{b}$, with \eqref{Order3} $\cdt_{a} \leqslant \cdt_{b}$, and thus we have $\cdt_{t}\leqslant \cdt_{u}$ and $\tdt_{t}=\tdt_{u}$, so $\boxed{(\cdt_{t},\tdt_{t})\moinslex (\cdt_{u},\tdt_{u})}$
		\item[+] else $\tdo_{b} >\tbg_{x}$, then $\cdt_{u}=\cdt_{b}+\cdt_{x}$ and $\tdt_{u}=\tdt_{x}$.
		\begin{itemize}
			\item[+] if $\cdo_{x}< \cdt_{x}$, then we have $\cdo_{t}<\cdo_{u}$ and thus $\boxed{(\cdt_{t},\tdt_{t})\moinslex (\cdt_{u},\tdt_{u})}$
			\item[+] else $\cdo_{x}=\cdt_{x}$, and as $q_{x} \in \monoidS$, \eqref{MonoidP2} ensures that we have $\tdo_{x} \leqslant \tdt_{x}$, thus $\tdt_{t}\leqslant \tdt_{u}$. As $\cdt_{t}\leqslant \cdt_{u}$, it comes $\boxed{(\cdt_{t},\tdt_{t})\moinslex (\cdt_{u},\tdt_{u})}$
		\end{itemize}
	\end{itemize}
	\item[+] else $\tdt_{a} >\tbg_{x}$, then $\cdt_{t} = \cdt_{a}+\cdt_{x}$ and $\tdt_{t}=\tdt_{x}$.
	\begin{itemize}
		\item[+] if $\tdt_{b} \leqslant\tbg_{x}$, then $\cdt_{u}=\cdt_{b}+\cdo_{x}$ and $\tdt_{u}=\tdo_{x}$. In this case we have $\tdo_{b}<\tdo_{a}$, and thus since $q_{a}\moins q_{b}$, \eqref{Order2} ensures that $\cdo_{a}<\cdo_{b}$.
		\begin{itemize}
			\item[+] if $\cdo_{x}=\cdt_{x}$, then $\cdt_{t}<\cdt_{u}$ and thus $\boxed{(\cdt_{t},\tdt_{t})\moinslex (\cdt_{u},\tdt_{u})}$.
			\item[+] else  $\cdo_{x}=\cdt_{x}-1$, and as $c_{x} \in \monoidS$, \eqref{MonoidP1} ensures that we have $\tdt_{x}\leqslant \tdo_{x}$, thus $\tdt_{t}\leqslant \tdt_{u}$. We have $\boxed{(\cdt_{t},\tdt_{t})\moinslex (\cdt_{u},\tdt_{u})}$
		\end{itemize}
		\item[+] else $\tdt_{b} >\tbg_{x}$, then $\cdt_{u}=\cdt_{b}+\cdt_{x}$ and $\tdt_{u}=\tdt_{x}$. As $q_{a}\moins q_{b}$, we are sure with \eqref{Order3} that $\cdt_{t}\leqslant \cdt_{u}$, and as $\tdt_{u}=\tdt_{t}$ we effectively have $\boxed{(\cdt_{t},\tdt_{t})\moinslex (\cdt_{u},\tdt_{u})}$
	\end{itemize}
    \end{itemize}
\end{itemize}
With the three above points, we have proved that $q_{a} \plus q_{x} \moins q_{b} \plus q_{x}$.\\

We are now going to show that $q_{x} \plus q_{a} \moins q_{x} \plus q_{b}$. We denote by $q_{t}=q_{x} \plus q_{a}$, and by $q_{u}=q_{x} \plus q_{b}$. Thus we want to prove that $q_{t}\moins q_{u}$.
\begin{itemize}
    \item Let us prove \eqref{Order1} i.e. $\tbg_{t} \geqslant \tbg_{u}$.
    We have $\tbg_{t}=\tbg_{x}$ and $\tbg_{u}=\tbg_{x}$, so we directly have $\boxed{\tbg_{t} \geqslant \tbg_{u}}$.
    \item Let us prove \eqref{Order2} i.e. $(\cdo_{t},\tdo_{t})\moinslex (\cdo_{u},\tdo_{u})$.
    \begin{itemize}
    	\item[+] if $\tdo_{x} \leqslant\tbg_{b}$, then $\tdo_{x} \leqslant\tbg_{a}$, and so we have : $\cdo_{u} = \cdo_{x}+\cdo_{b}$, $\tdo_{u}=\tdo_{b}$, $\cdo_{t} = \cdo_{x}+\cdo_{a}$ and $\tdo_{t}=\tdo_{a}$. So as $q_{a} \moins q_{b}$, we directly have $\boxed{(\cdo_{t},\tdo_{t})\moinslex (\cdo_{u},\tdo_{u})}$.
    	\item[+] else $\tdo_{x}>\tbg_{b}$, then  $\cdo_{u} = \cdo_{x}+\cdt_{b}$ and $\tdo_{u}=\tdt_{b}$.
    	\begin{itemize}
    		\item[+] if $\tdo_{x}\leqslant\tbg_{a}$, then $\cdo_{t} = \cdo_{x} + \cdo_{a}$. Since $q_{a} \in \monoidS$, we have $\cdo_{a}\leqslant\cdt_{a}$, and since $q_{a} \moins q_{b}$, \eqref{Order2} ensures that $\cdt_{a}\leqslant\cdt_{b}$. Thus we have $\cdo_{t}\leq \cdo_{u}$.
    		\begin{itemize}
    			\item[+] if $\cdo_{a} < \cdt_{b}$, then $\cdo_{t} <\cdo_{u}$ and thus $\boxed{(\cdo_{t},\tdo_{t})\moinslex(\cdo_{u},\tdo_{u})}$.
    			\item[+] else $\cdo_{a} = \cdt_{b}$, and thus $\cdo_{a} = \cdt_{a}=\cdt_{b}$. Since $q_{a} \in \monoidS$, with converse of \eqref{MonoidP2}, it comes that $\tdo_{a} \leqslant\tdt_{a}$. Moreover as $q_{a} \moins q_{b}$ and $\cdt_{a}=\cdt_{b}$, we have $\tdt_{a}\leqslant\tdt_{b}$. By transitivity  $\tdo_{a}\leqslant\tdt_{b}$, which proves that $\boxed{(\cdo_{t},\tdo_{t})\moinslex(\cdo_{u},\tdo_{u})}$.
    		\end{itemize}
    		\item[+] else $\tdo_{x}> \tbg_{a}$, then $\cdo_{t} = \cdo_{x} + \cdt_{a}$ and $\tdo_{t}=\tdt_{a}$. Since $q_{a} \moins q_{b}$, we have $\boxed{(\cdo_{t},\tdo_{t})\moinslex (\cdo_{u},\tdo_{u})}$.
    	\end{itemize}
    \end{itemize}

\item Let us prove \eqref{Order3} i.e. $(\cdt_{t},\tdt_{t})\moinslex (\cdt_{u},\tdt_{u})$.
\begin{itemize}
	\item[+] if $\tdt_{x} \leqslant \tbg_{b}$, then $\tdt_{x} \leqslant \tbg_{a}$, and so we have : $\cdt_{u} = \cdt_{x}+\cdo_{b}$, $\tdt_{u}=\tdo_{b}$, $\cdt_{t} = \cdt_{x}+\cdo_{a}$ and $\tdt_{t}=\tdo_{a}$. So as $q_{a} \moins q_{b}$, we directly have $\boxed{(\cdt_{t},\tdt_{t})\leq (\cdt_{u},\tdt_{u})}$.
	\item[+] else $\tdt_{x}>\tbg_{b}$, then  $\cdt_{u} = \cdt_{x}+\cdt_{b}$ and $\tdt_{u}=\tdt_{b}$.
	\begin{itemize}
		\item[+] if $\tdt_{x}\leqslant \tbg_{a}$, then $\cdt_{t} = \cdt_{x} + \cdo_{a}$ and $\tdt_{t}=\tdo_{a}$. Since $q_{b} \in \monoidS$, we have $\cdo_{b}\leqslant \cdt_{b}$, and since $q_{a} \moins q_{b}$, \eqref{Order2} ensures that $\cdo_{a}\leq \cdo_{b}$. So we have $\cdo_{a}\leq \cdo_{b} \leq \cdt_{b}$ and $\cdt_{t}\leq \cdt_{u}$.
		\begin{itemize}
			\item[+] if $\cdo_{a} < \cdt_{b}$, then $\cdo_{t} <\cdo_{u}$ and thus $\boxed{(\cdo_{t},\tdo_{t})\moinslex (\cdo_{u},\tdo_{u})}$.
			\item[+] else $\cdo_{a} = \cdt_{b}$, and thus $\cdo_{a} = \cdo_{b}=\cdt_{b}$. As $q_{b} \in \monoidS$, with converse of \eqref{MonoidP2} it comes that $\tdo_{b} \leqslant \tdt_{b}$. Moreover as $q_{a} \moins q_{b}$ and $\cdo_{a}=\cdo_{b}$, we have with \eqref{Order2} $\tdo_{a}\leqslant \tdo_{b}$. By transitivity  $\tdo_{a}\leqslant \tdt_{b}$, which proves that $\boxed{(\cdo_{t},\tdo_{t})\moinslex (\cdo_{u},\tdo_{u})}$.
		\end{itemize}
		\item[+] else $\tdt_{x}> \tbg_{a}$, then $\cdt_{t} = \cdt_{x} + \cdt_{a}$ and $\tdt_{t}=\tdt_{a}$. Since $q_{a} \moins q_{b}$, we have $\boxed{(\cdt_{t},\tdt_{t})\moins (\cdt_{u},\tdt_{u})}$.
	\end{itemize}
\end{itemize}
\end{itemize}
With the three above points, we have proved that $q_{x} \plus q_{a} \moins q_{x} \plus q_{a}$.\\

Finally we have proved both $q_{x} \plus q_{a} \moins q_{x} \plus q_{b}$ and $q_{a} \plus q_{x} \moins q_{b} \oplus q_{x}$, thus the addition is compatible with the order.
\end{proof}

\medbreak

\begin{proof}[Proof of lemma~\ref{lemma:Meet}.] In this proof we use notations established in proofs of Lemmas~\ref{lemma:Monoid} and~\ref{lemma:OrderCompatible}. 
We first prove by disjunction of cases the stability of $\meet$. It will then follow that $\meet$ defines a greatest lower bound for any pair of elements belonging to $\monoidS$.
Let $q_{a},q_{b} \in \monoidS,\enskip q= q_{a} \meet q_{b}$. We are going to show that $q \in \monoidS$. If $q_{a} \in \{e,\infty\}$ or $q_{b} \in  \{e,\infty\}$, then by definition of $\moins$ and $\meet$, the result is immediate. 
Suppose now that let $q_{a} = \res[a]$, $q_{b} = \res[b]$ and $q= \res$.\\

\begin{itemize}
    \item We first prove that $q$ satisfies \eqref{MonoidP1}. We assume that we have $\tdo < \tdt$.
\begin{itemize}
	\item[+] if $(\tdo =\tdo_{a} \text{ and } \tdt =\tdt_{a})$ or $(\tdo =\tdo_{b} \text{ and } \tdt =\tdt_{b})$ , then it implies $(\cdo =\cdo_{a} \text{ and } \cdt =\cdt_{a})$ or $(\cdo =\cdo_{b} \text{ and } \cdt =\cdt_{b})$. We use \eqref{MonoidP1} with $q_{a}$ or $q_{b}$ and we directly have that $\boxed{\cdo = \cdt}$.
	\item[+] if  $(\tdo =\tdo_{a} \text{ and } \tdt =\tdt_{b})$, then by definition of the meet $\cdo = \cdo_{a} \leqslant \cdo_{b}$ and $\cdt = \cdt_{b} \leq \cdt_{a}$.
	\begin{itemize}
		\item[+] if $\cdt_{a}=\cdt_{b}$, then $\tdt_{b} \leqslant \tdt_{a}$ by definition of the order with \eqref{Order3}. As $\tdo < \tdt$, it comes that $\tdo_{a} < \tdt_{a}$, and as $q_{a} \in \monoidS$ we can use \eqref{MonoidP1} to show that $\cdo_{a}=\cdt_{a}$. Thus we have $\boxed{\cdo = \cdt}$.
		\item[+] else $\cdt_{a} >\cdt_{b}$, and so we have the followings inequalities : $\cdo_{a}\leq \cdo_{b} \leq \cdt_{b} < \cdt_{a}$. Since $q_{a} \in \monoidS$ and $\cdo_{a} <\cdt_{a}$, it comes that $\cdo_{a}=\cdt_{a} -1 $ which implies that $\cdo_{a}=\cdo_{b}=\cdt_{b}$ and thus $\boxed{\cdo = \cdt}$.
	\end{itemize}
	\item[+] else $(\tdo =\tdo_{b} \text{ and } \tdt =\tdt_{a})$, and we conclude by symmetry of $q_{a}$ and $q_{b}$.
\end{itemize}

\item We nos prove that $q$ satisfies  \eqref{MonoidP2}. We assume that we have $\tdt < \tdo$.
\begin{itemize}
	\item[+] if $(\tdo =\tdo_{a} \text{ and } \tdt =\tdt_{a})$ or $(\tdo =\tdo_{b} \text{ and } \tdt =\tdt_{b})$, then it implies $(\cdo =\cdo_{a} \text{ and } \cdt =\cdt_{a})$ or $(\cdo =\cdo_{b} \text{ and } \cdt =\cdt_{b})$. We use \eqref{MonoidP1} with $q_{a}$ or $q_{b}$ and we directly have that $\boxed{\cdo = \cdt-1}$.
	\item[+] if  $(\tdo =\tdo_{b} \text{ and } \tdt =\tdt_{a})$, then by definition of the meet $\cdo = \cdo_{b} \leqslant \cdo_{a}$ and $\cdt = \cdt_{a} \leqslant \cdt_{b}$.
	\begin{itemize}
		\item[+] if $\cdo_{a}=\cdo_{b}$, then $\tdo_{b} \leqslant \tdo_{a}$. As $\tdt < \tdo$, it comes that $\tdt_{a} < \tdo_{a}$, and as $q_{a} \in \monoidS$ we can use \eqref{MonoidP2} to show that $\cdo_{a}=\cdt_{a}-1$. Thus we have $\boxed{\cdo = \cdt-1}$.
		\item[+] else $\cdo_{a} >\cdo_{b}$, and so we have the inequalities : $\cdo_{b}< \cdo_{a}\leq \cdt_{a} \leq \cdt_{b}$. $q_{b}, q_{a} \in \monoidS$, thus it comes that $\cdo_{b}=\cdt_{b} -1 $ and $\cdo_{a}=\cdt_{a}$, which implies that $\cdo_{b}=\cdo_{a}-1=\cdt_{a}-1=\cdt_{b}-1$ and thus $\boxed{\cdo = \cdt-1}$.
	\end{itemize}
	\item[+] else $(\tdo =\tdo_{a} \text{ and } \tdt =\tdt_{b})$, and we conclude by symmetry of $q_{a}$ and $q_{b}$.
\end{itemize}

\item We finally prove that $q$ satisfies  \eqref{MonoidP3}. We assume we have $\tdo = \tdt$.
First, as by construction of the meet and the fact that $q_{a}, q_{b} \in \monoidS$ we have $\cdo \leqslant \cdo_{a} \leqslant \cdt_{a}$, $\cdo \leqslant \cdo_{b} \leqslant \cdt_{b}$ and $\cdt =\cdt_{a}$ or  $\cdt = \cdt_{b}$. Thus we directly have : $\cdo\leqslant \cdt$.
\begin{itemize}
	\item[+] if $\cdo = \cdt$, \eqref{MonoidP3} is satisfied.
	\item[+] else $\cdo < \cdt$.
	\begin{itemize}
		\item[+] if $(\cdo = \cdo_{a} \text{ and } \cdt = \cdt_{a})$ or $(\cdo = \cdo_{b} \text{ and } \cdt = \cdt_{b})$, then result comes directly from the fact that $q_{a} \in \monoidS$ or $q_{b} \in \monoidS$.
		\item[+] else if $\cdo = \cdo_{a}$ and $\cdt=\cdt_{b}$, then by definition of the meet $\cdo=\cdo_{a} \leqslant \cdo_{b}$ and $\cdt=\cdt_{b}\leqslant \cdt_{a}$. So $\cdo_{a}<\cdt_{b}\leqslant\cdt_{a}$ and thus $\cdo_{a}=\cdt_{a}-1$ and $\cdt_{b}=\cdt_{a}$. So it finally comes that $\cdt=\cdt_{a}$ and as $\cdo=\cdo_{a}$ we have $\cdo=\cdt-1$ and property \eqref{MonoidP3} is satisfied.
		\item[+] else $\cdo=\cdo_{b}$ and $\cdt=\cdt_{a}$, and we conclude by symmetry of $q_{a}$ and $q_{b}$.
	\end{itemize}
\end{itemize}
\end{itemize}

All properties of $\monoidS$ are satisfied, proving the stability of $\meet$.
\\
It then follows from the definitions of the order $\moins$ and of the meet $\meet$ that $q$ is the greatest lower bound on $q_a$ and $q_b$.
\end{proof}

\medbreak

\begin{proof}[Proof of Proposition~\ref{prop:FinalCost}.]
Let $\shift$ be a shift and $P$ the origin-destination path associated. As seen with the proposition~\ref{prop:bijectionDigraph}, $P$ always exists. We denote $q_{P} =((q_{P}^{\oscenario})_{\oscenario\in \Oscenario},l)$ its resource.\\
Let us prove first that $l = \cw(\he[\sh] - \hb[\sh]) - \sum_{j \in \sh} \lambda_j $. By definition of the monoid $(\bbR,+)$, $l$ is the sum of the second parts of each arc resource. Since for any job $j$ we initialize an arc of the form $\big((j,\hb,\cdot),(j',\hb,\cdot)\big)$ with a resource $((q_{a}^{\oscenario})_{\oscenario \in \Oscenario},-\lambda_j)$, and an arc of the form $\big((j,\hb,\cdot),\he\big)$ with a resource $((q_{a}^{\oscenario})_{\oscenario \in \Oscenario},\cw(\he - \hb) - \lambda_j)$, and since every $o$-$d$ path contains exactly one arc of the form $\big((j,\hb,\cdot),\he\big)$ and one vertex of the form $(j,\hb,\cdot)$ for every $j \in \sh$, then we obtain $l = \cw(\he[\sh] - \hb[\sh]) - \sum_{j \in \sh} \lambda_j $.\\

\noindent Let us now prove by induction on the number $n$ of arcs in $P$ that,

\begin {table}[h]

\begin{tabular}{m{14cm}m{1cm}}
	for every $\oscenario \in \Oscenario$, resource $q_{P}^{\oscenario} = \res[\oscenario]$, where $\tbg_{P}(\omega)$ is the time at which the first job of $\sh$ starts under $\omega$,
$\tdo_{P}(\omega)$ the time at which the last job of $\sh$ under $\omega$ ends if non-rescheduled, else has value $-\infty$,
$\tdt_{P}(\omega)$ the time at which the last job of $\sh$ under $\omega$ would have ended if the first job of $\sh$ was rescheduled under $\omega$, and if that job is non-rescheduled, else has value $-\infty$,
$\cdo_{P}(\omega)$ is the number of rescheduled job in $\sh$ under~$\omega$,
and $\cdt_{P}(\omega)$ the number of jobs of $\sh$ that would have been rescheduled if the first job of $\sh$ was rescheduled under $\omega$. &(*)
\end{tabular}
\end {table}

\noindent\textbf{Initialization}: the result for $n= 1$ is immediate from the definition of arc resources. \\
\textbf{Legacy}: Assume (*) true for paths of length $n$. Consider a path $P$ of length n+1 and denote $\sh$ the associated shift or sequence of tasks. $P$ can be decomposed as $Q + a$, with $Q$ a path of length n and $a$ a single arc.\\
Let consider a scenario $\oscenario \in \oscenario$.\\
We have $\tbg_{q}(\oscenario) = \tbg_{P}(\oscenario)$, and by hypothesis, (*) is true for path $Q$, thus $\tbg_{q}(\oscenario)$ is the time at which the first job of path $Q$ starts under $\oscenario$, which is also the first job of path $P$.\\
If the succession of $Q$ with $a$ is possible, i.e. $\tdo_{q}(\oscenario) \leqslant \tbg_{a}(\oscenario)$, then $\cdo_{P}(\oscenario) = \cdo_{q}(\oscenario) + \cdo_{a}(\oscenario)$, and $\tdo_{P}(\omega) = \tdo_{a}(\omega)$.
As (*) is true for $Q$, as $a$ is an arc, and as the last job of $P$ is the job represented by arc $a$:
\begin{itemize}
    \item if $\tdo_{a}(\omega) = \xe[j](\omega)$, then $\tdo_{P}(\omega) = \xe[j](\omega)$ which corresponds to the case where $\sh$ terminates with job $j$
    \item if $\tdo_{a}(\omega) = \xe[j](\omega) + \tbr$, then $\tdo_{P}(\omega) = \xe[j](\omega)$ which corresponds to the case where $\sh$ terminates with a break
    \item else $\tdo_{a}(\omega) = -\infty$, then $\tdo_{P}(\omega) = -\infty$, which corresponds to the case where the last job is very late
\end{itemize} 
In all cases, (*) is true in what concerns $\tdo_{P}(\omega)$ since the last job of $P$ is the one from arc $a$.\\
Concerning the cost:\begin{itemize}
    \item if $\vl[j](\omega) = 1$, then $\cdo_{a}(\oscenario) =1$ and thus $\cdo_{P}(\oscenario)$ is the number of rescheduled job in $\sh$ under $\oscenario$ since arc $a$ represents an always rescheduled job.
    \item else $\vl[j](\omega) = 0$, then $\cdo_{a}(\oscenario) =0$ and thus $\cdo_{P}(\oscenario)$ is the number of rescheduled job in $\sh$ under $\oscenario$, since the succession of $Q$ with $a$ is possible and thus does not generated any reschuled cost.
\end{itemize} 
We easily show in the same way that $\tdt(\omega)$ is the time at which the last job of $\sh$ under $\omega$ would have end if the first job of $\sh$ was rescheduled under $\omega$ and if the last job is non-rescheduled, else is $-\infty$, and $\cdt_{P}(\omega)$ the number of jobs of $\sh$ that would have been rescheduled if the first job of $\sh$ was rescheduled under $\omega$.\\
Else the succession of $Q$ with $a$ is not possible, i.e. $\tdo_{q}(\omega)  > \tbg_{a}(\omega)$. As (*) is true for $Q$ and as $a$ is an arc, by using the same argumentss than previously regarding initialization of resource on the arcs in each possible case, we conclude that (*) is true for any path $P$ of length n+1.
\\
Since any $o$-$d$ path $P$ has a finite number of vertices, and since (*) is true for every n, its concludes the proof.
\end{proof}
\medbreak
\begin{proof}[Proof of Theorem~\ref{prop:FinalProp}.]
The bijection is ensured by Proposition~\ref{prop:bijectionDigraph}. \\
Let $q_{P} = \left(\bigoplus_{a \in P}r_a\right)$. With Proposition~\ref{prop:FinalCost}, we have $q_{P} = \left(
\left(\begin{array}{c}
            \mathsf{bg}\big(\tbg_{P}(\omega)\big)\\
            \mathsf{do}\big(\cdo_{P}(\omega),\tdo_{P}(\omega)\big)\\
            \mathsf{dt}\big(\cdt_{P}(\omega),\tdt_{P}(\omega)\big)
    \end{array}\right)_{\oscenario \in \Oscenario}
, \cw(\he[\sh] - \hb[\sh]) - \sum_{j \in \sh} \lambda_j \right)$. where $[\hb[\sh],\he[\sh]]$ corresponds to the time interval of shift $\sh$ and $\cdo_{P}(\omega)$ is the number of rescheduled job in $\sh$ under $\omega$.\\
By definition of the cost function: $c_{\monoidS}\left(\left(\begin{array}{c}
            \mathsf{bg}\big(\tbg_{P}(\omega)\big)\\
            \mathsf{do}\big(\cdo_{P}(\omega),\tdo_{P}(\omega)\big)\\
            \mathsf{dt}\big(\cdt_{P}(\omega),\tdt_{P}(\omega)\big)
    \end{array}\right)\right)= \cdo_{P}(\oscenario)$, and thus with the definition of the cost function on a resource, we obtain
    $$ c\left(\left(\bigoplus_{a \in P}r_a\right)\right)=\frac{\alpha * \cbu}{\sharp \oscenario}\sum_{\oscenario \in \oscenario}c_{\monoidS}(q^{\oscenario}) + \cw(\he[\sh] - \hb[\sh]) - \sum_{j \in \sh} \lambda_j = \costshift - \sum_{\task \in \shift} \lambda_{\task}$$
which concludes the Theorem~\ref{prop:FinalProp}.
\end{proof}

\end{appendix}

\end{document}